\documentclass[letterpaper,12pt,draftcls,onecolumn,romanappendices]{IEEEtran}
\IEEEoverridecommandlockouts

\usepackage[margin=0.75in]{geometry}
\usepackage{amsmath,amssymb, theorem}
\usepackage{tikz}
\usepackage{epstopdf}
\usepackage{upgreek}
\usepackage{subcaption}
\usetikzlibrary{automata,positioning,fit,shapes.geometric}
\pgfdeclarelayer{bg}    
\pgfsetlayers{bg,main}  
\pgfdeclarelayer{background}
\pgfdeclarelayer{foreground}
\pgfsetlayers{background,main,foreground}
\usepackage{ifthen}
\usepackage{psfrag}
\usepackage{cite}
\usepackage{graphicx}
\usepackage{graphics}
\usepackage{epstopdf}
\usepackage{bbm}
\usepackage{float}
\usepackage{framed}
\usepackage{textcomp}
\usepackage{mathabx}
\usepackage{accents}
\restylefloat{figure}
\usepackage{color,xspace}
\usepackage{tikz}
\usetikzlibrary{decorations.pathreplacing}
\usepackage{upgreek}
\usetikzlibrary{automata,positioning}

\usepackage{pstool}
\usepackage{url}

\newcommand{\real}{\mathbb{R}}

\newcommand{\dom}{\operatorname{dom}}

\newcommand{\goesto}{\rightarrow}

\newcommand{\rateit}[1]{\lambda^{\{#1\}}_{i \; (t)}}

\newcommand{\ubar}[1]{\underaccent{\bar}{#1}}
\newcommand{\utilde}[1]{\underaccent{\tilde}{#1}}
\newcommand{\comps}{\mathcal{C}}

\newtheorem{lem}{\textbf{Lemma \newline}}
\newtheorem{theorem}{\textbf{Theorem \newline}}

\newtheorem{cor}{\textbf{Corollary \newline}}

\newcommand{\oprocendsymbol}{\hbox{$\bullet$}}
\newcommand{\oprocend}{\relax\ifmmode\else\unskip\hfill\fi\oprocendsymbol}

\newenvironment{proof}{\noindent \textbf{Proof :\newline}}{ \oprocend}
{}

\newtheorem{remark}[theorem]{Remark}

\begin{document}
\title{Bounding Approximations for Compartmental Spreading Processes and Applications to Control}
\author{Nicholas J. Watkins, Cameron Nowzari,  Victor M. Preciado, and George J. Pappas\thanks{The authors are with the Department of Electrical and Systems Engineering, University of Pennsylvania, Pennsylvania, PA 19104, USA, {\tt\small \{nwatk,cnowzari,preciado,pappasg\}@upenn.edu}}}
      
\title{Deterministic Bounding Systems for Stochastic Compartmental Spreading Processes}      
      
\maketitle
\begin{abstract}
	This paper studies a novel approach for approximating the behavior of compartmental spreading processes.  In contrast to prior work, the methods developed describe a dynamics which \emph{bound} the exact moment dynamics, without explicitly requiring \emph{a priori} knowledge of non-negative (or non-positive) covariance between pairs of system variables.  Moreover, we provide systems which provide \emph{both} upper- and lower- bounds on the process moments.  We then show that when system variables are shown to be non-negatively (or non-positively) correlated for all time in the system's evolution, we may leverage the knowledge to create \emph{better} approximating systems.  We then apply the technique to several previously studied compartmental spreading processes, and compare the bounding systems' performance to the standard approximations studied in prior literature.
\end{abstract}

\section{Introduction}
\subsection{Motivation, Background, and Contributions}
Many systems of interest - viral epidemics, belief propagations, chemical reactions - can be modeled as continuous-time Markov processes with moments that evolve in exponentially large state spaces.  Due to their poor scaling properties, analysis of the exact moment dynamics is prohibitive; methods of approximation must be invoked so as to represent the systems in a space which grows tractably with the size of the problem instance.  Commonly, techniques referred to as \emph{moment closure} approximations are used.  

While  moment closure methods are popular and often claimed as accurate, analytic results are typically restricted to meager statements, e.g. demonstrating that the exact dynamics are approximated well within a neighborhood of the initial point.  Results guaranteeing that the moments of the process remain close to the moments given by the approximated system often rely on simulation, and must be taken with caution: we begin to address this problem in our work here.

This paper develops a method by which appropriate processes can be approximated by a \emph{bounding system} - one equipped with states which provably upper- and lower-bound the exact system's moments.  For purposes of concreteness, we focus on application to a general class of epidemic spreading processes: compartmental spreading models.  The model we study here subsumes many of the epidemic models studied in literature, and so we believe to be of immediate use to the community of researchers currently engaged in the area.

\paragraph*{Literature Review}
Moment closure techniques find wide use in the analysis of high-dimensional Markov processes, which are natural model choices in a variety of contexts.  In chemistry, reactions which are modeled stochastically via Markov processes give rise to the Chemical Master Equations (CMEs) \cite{gillespie1992rigorous}, of which approximations are commonly studied \cite{singh2011approximate}.  Similar models occur in biology, wherein populations are modeled stochastically \cite{singh2006moment}; in genetics, wherein protein creation and destruction are modeled stochastically \cite{soltani2014moment}; and in epidemiology, wherein the spread of disease is modeled stochastically \cite{sahneh2013generalized}.  

A variety of closure techniques have been used.  At root, they divide into two categories: those which force assumptions on the process variables - such as normality (see, e.g., \cite{soltani2014moment}), conditional normality (see, e.g., \cite{soltani2014moment}), and independence (see, e.g., \cite{soltani2014moment,van2009virus}) - and those which construct derivative-matching systems near the initial system state \cite{singh2011approximate}.  The analytic guarantees of accuracy for such methods are not appropriate for problems concerning control or network design.  Typical results only guarantee that the approximate system's states are bounded in error over a (possibly infinitesimal) interval after the initial time \cite{singh2011approximate}.  Approximation schemes exist which guarantee an error bound, such as the finite-state projection algorithm \cite{munsky2006finite,munsky2008finite}, but require the repetitive solution of a system of ordinary differential equations to execute and only provide guarantees \emph{after} the solution is computed, and so may be of limited use in control applications.  Moreover, to the authors' knowledge, none of these schemes provide guarantees as to the nature of the error: the approximate systems are not generally guaranteed to over-approximate or under-approximate the exact system's moments. 

The ability to claim a rigorous upper- or lower-bound on the evolution of the process moments is useful for control applications.  In particular, if our objective is to control an epidemic to extinction at a guaranteed rate, or ensure that a particular behavior survives for (or dies out within) a given amount of time, we must have known bounds on the moments of the process.  Moreover, in order to claim control of an approximate model (such as done in the recent works \cite{preciado2014optimal,CN-VMP-GJP:15-TCNS,watkins2015optimal}) implies control of the modeled process, we must have bounding guarantees - approximations alone do not suffice. 

In compartmental spreading models, an assumption of independence is often applied (see, e.g., \cite{sahneh2013generalized}) and is referred to as a \emph{mean-field} approximation.  For particular choices of process structure (e.g. Susceptible-Infected-Susceptible $SIS$ \cite{van2009virus} and Susceptible-Infected-Removed $SIR$ \cite{van2014exact}), it can be shown that the random variables involved are non-negatively correlated \cite{cator2014nodal}.  This supports a claim that the engendered approximation is an upper-bound on the infection rate of the process \cite{van2009virus}, though further claims of this type are lacking for other processes.  In this paper, we establish a rigorous condition for testing whether a particular system is a \emph{bounding system} for a given process, provide a construction for a non-trivial bounding system which is valid for \emph{any} compartmental spreading process (regardless of the degree of correlation between the involved variables), and show that whenever knowledge of correlation between variables is available, it may be used to construct better bounding systems.  We then apply our results to several models previously studied in the literature, and comment on the performance of the bounding systems to the approximations of prior work. 

\paragraph*{Statement of Contributions}
We define a notion of \emph{bounding systems} as a technique for approximating stochastic processes which cannot be tractably analyzed exactly.  For purposes of concreteness, we center our analysis on a general class of compartmental spreading models, similar to that which a general method of mean-field approximations was developed \cite{sahneh2013generalized}.  In particular, we add the notion of ``affector sets,'' to the models studied in \cite{sahneh2013generalized}: these allow us to cleanly consider models in which a particular transition is affected by neighboring nodes in a set of specified compartments.  We show that for any model belonging to this class, a non-trivial bounding system may be constructed (Section \ref{subsec:construct}), and that whenever knowledge of the involved variables' correlation is known \emph{a priori}, it may be incorporated to form bounding systems which incur less error (Section \ref{sec:bounding}).  We demonstrate the utility of our results by applying them to standard epidemic models in (Section \ref{sec:app}).  We demonstrate that whereas for some processes (e.g. Susceptible-Infected-Susceptible), the standard mean-field approximations are recovered trajectories of the bounding systems, there exist processes (e.g. Susceptible-Infected 1-Susceptible-Infected 2- Susceptible) for which this not the case.

\subsection{Preliminaries}
\paragraph*{Dynamical systems}
We denote the $n$-dimensional Euclidean space as $\real^n$. Given a vector $x \in \real^n$, we write its Euclidean norm as $\|x\|$.  We denote a system dynamics by the state equation $\dot{x} = f(x)$, where $f:\dom(f) \mapsto \real^n$, and the trajectory generated by the system as $x(t)$, where we omit the dependence on time wherever it is clear from context.  We say a function is globally Lipschitz continuous if there exists some constant $L$ such that $f(x) - f(y) \leq L \|f(x) - f(y)\|$ for all $x$ and $y$ in the domain of $f$. The cardinality of a set $Z$ is denoted by~$|Z|$.

\paragraph*{Probability}
We denote the probability of an event $A$ by $\Pr(A)$, the expectation of a random variable $X$ by $\mathbb{E}[X]$, and the indicator function of an event $A$ by $\mathbbm{1}_{\{A\}}$.  The covariance of two random variables $X$ and $Y$ is defined by $\sigma(X,Y) = \mathbb{E}[XY] - \mathbb{E}[X] \mathbb{E}[Y]$. 

\paragraph*{Graph theory}
A graph $G$ is defined as a pair $\{V,E\}$, in which $V$ is a set of vertices, and $E$ is a set of edges.  Two vertices $x,y \in V$ are connected if and only if $(x,y) \in E$.

\section{Compartmental Spreading Process Model}
\subsection{Spreading Process Construction} \label{subsec:construct}
We consider a Markovian spreading process in which each agent $i$ in the spreading graph $G = \{V,E\}$ takes membership in some compartment $c \in \mathcal{C}$ at all times.  We will denote the membership of node $i$ at time $t$ by $x_i(t) = c$.  Transitions from membership in compartment $c$ to membership in compartment $c^\prime$ (denoted by $\{c \rightarrow c^\prime\}$) occur stochastically via Poisson processes with rates $\rateit{c \rightarrow c^\prime }$, where

\begin{align*}
\lambda^{\{c \rightarrow c^\prime\}}_{i \; (t)} 
	= \begin{cases} 
	\sum_{\{j \in V,\; \tilde{c} \in \mathcal{A} (c,c^\prime)\}} \mathbbm{1}_{\{x_{j}(t) = \tilde{c} \}} \beta^{\{\tilde{c}; \, c \rightarrow c^\prime\}}_{ij} \ & \text{ if } \{c \rightarrow c^\prime\} \text{ is an external process},\\
                   \delta_i^{\{c \rightarrow c^\prime \}} & \text{ if } \{c \rightarrow c^\prime\} \text{ is an internal process}, \\
                   0 & \text{ if } \{c \rightarrow c^\prime\} \text{ is not possible},
      \end{cases}
\end{align*}
and we use the notation $\mathbbm{1}_{\{\cdot\}}$ as the indicator of an event, $\beta^{\{\tilde{c}; \, c \rightarrow c^\prime\}}_{ij}$ as the rate at which node $j$ in compartment $\tilde{c}$ generates events which transition node $i$ from membership in compartment $c$ to membership in compartment $c^\prime$, and $\delta_i^{\{c \rightarrow c^\prime \}}$ as the rate at which node $i$ generates events which transition node $i$ from membership in compartment $c$ to membership in compartment $c^\prime$.  Note that we define a process as \emph{external} whenever the transition $\{c \goesto c^\prime \}$ at node $i$ at time $t$ is influenced by the compartment membership of some node $j$ at time $t$.  A process is \emph{internal} whenever the process is not external, i.e. when the transition at a node $i$ is not affected by the compartment membership any other node $j$.  We define the set $\mathcal{A}(c,c^\prime)$ as the \emph{affector} set associated with the external transition $\{c \goesto c^\prime \}$, that is the set of all compartments by which the transition $\{c \goesto c^\prime \}$ is affected.  

The dynamics of the first moment (i.e., the expectation, denoted $\mathbb{E}$) of this class of processes can be written as:
\begin{equation} \label{eq:dynamics}
	\begin{aligned}
	\frac{d \mathbb{E} \left[ \mathbbm{1}_{\{x_{i}(t) = \{c\} \}} \right]}{d t} 
	&= \mathbb{E} 
		 \left[ \sum_{c^\prime \in \comps \setminus \{c\} } \mathbbm{1}_{\{x_{i}(t) = c^\prime\}} \rateit{c^\prime \rightarrow c} \right]
		- \mathbb{E} \left[ \sum_{c^\prime \in \comps \setminus \{c\}} \mathbbm{1}_{\{x_{i}(t) = c \}} \rateit{c \rightarrow c^\prime} \right],
	\end{aligned}
\end{equation}

and are well defined under mild assumptions, e.g. that $0 \leq \delta_i^{\{c \goesto c^\prime\}} < \infty$ and
 $0 \leq \beta_{ij}^{\{\tilde{c}; \; c \goesto c^\prime\}} < \infty$ for all $i$, $j$, $\tilde{c}$, $c$ and $c^\prime$.  We will not pause to derive these equations here, though we provide an extended discussion of their construction in Appendix \ref{app:construction}.  For further details, the interested reader may consult a reference of Markov processes, such as \cite{stroock2005introduction}.  
 
 This representation of the system is agent-centric: it considers states as the compartment membership probabilities of each node.  We choose to consider this representation due to its size: it is $O(|\comps|n)$ dimensional.  We note that this representation is not unique: we may represent the moment dynamics as a linear system in $O(|\comps|^n)$ dimensions if we consider a state space in which each possible node-compartment combination for the process is a state - this is described in much greater detail for a similar class of spreading models in \cite{sahneh2013generalized}.  This approach is intractable for all but the smallest number of nodes.

Whenever all processes are internal, the dynamics \eqref{eq:dynamics} are closed (i.e. each variable has a corresponding equation in the system dynamics), and may be analyzed without approximation: this representation is exact, which will be shown in Section \ref{subsec:internal}. 
However, when a process $\{c \rightarrow c^\prime\}$ is external, second-order moments of the form $\mathbb{E} [\mathbbm{1}_{\{x_{i}(t) = c \}} \mathbbm{1}_{ \{ x_{j}(t) = \tilde{c} \}}]$ are introduced, preventing closure: we do not have equations in the system dynamics which describe the evolution of $\mathbb{E} [\mathbbm{1}_{\{x_{i}(t) = c \}} \mathbbm{1}_{ \{ x_{j}(t) = \tilde{c} \}}]$.  Since higher order moments cannot, in general, be expressed in terms of first-order moments without incurring error, we cannot resolve this issue without approximation or a change in representation.  This is primarily addressed in one of two ways: representing the moment dynamics as a linear system in a larger ($O(|\comps|^n)$) state space (see \cite{sahneh2013generalized} for details), or applying a moment closure technique.  As the former is intractable, we shall concern ourselves with the latter.

\begin{figure}
\centering
\begin{tikzpicture}[thick,every node/.style={minimum size=.5em},decoration={border,segment length=2mm,amplitude=0.3mm,angle=90}]
				\pgfdeclarelayer{background}
				\pgfdeclarelayer{foreground}
				\pgfsetlayers{background,main,foreground}
  \node[draw,circle,fill=white, thick] (1) at (-1,1) {a};
  \node[draw,circle,fill=white, thick] (2) at (.5,1) {b};
  \node[draw,circle,fill=white, thick] (3) at (2,1) {c};
  \node[draw,circle,fill=white, thick] (4) at (-1,-1) {d};
  \node[draw,circle,fill=white, thick] (5) at (.5,-1) {e};
  \node[draw,circle,fill=white, thick] (6) at (2,-1) {f};
  \node[draw=black,dashed,fit=(1) (4) ,inner sep=0.1ex,ellipse] (el1) {};
  \node[draw=black,dashed,fit=(6) ,inner sep=0.1ex,ellipse] (el2) {};
  
  \path[->] (1) edge [very thick] (2);
  \path[->] (1) edge [very thick, bend left] (3);
  \path[->] (2) edge [very thick] (3);
  \path[->] (3) edge [very thick,dashed] (6);
  \path[->] (3) edge [very thick] (4);
  \path[->] (4) edge [very thick] (5);
  \path[->] (5) edge [very thick] (6);
  \path[->] (1) edge [very thick, bend left] (4);
  \path[->] (4) edge [very thick, bend left] (1);
  \path[->] (6) edge [very thick, bend left,dashed] (4);
  
  \node[draw,circle,fill=gray, thick] (21) at (5.5,2) {};
  \node[draw,circle,fill=gray, thick] (22) at (4.5,1.1) {};
  \node[draw,circle,fill=gray, thick] (23) at (7.5,1.5) {};
  \node[draw,circle,fill=gray, thick] (24) at (5,-1.3) {};
  \node[draw,circle,fill=gray, thick] (25) at (4,0) {};
  \node[draw,circle,fill=gray, thick] (26) at (6.1,1.4) {};
  \node[draw,circle,fill=gray, thick] (27) at (5.5,-0.2) {};
  \node[draw,circle,fill=gray, thick] (28) at (6,-1.1) {};
  \node[draw,circle,fill=gray, thick] (29) at (7.1,-.5) {}; 
  
  \begin{pgfonlayer}{background}
	\draw[black, thick, dotted] (3,-1.15) -- (3.8,0);
	\draw[black, thick, dotted] (3,1.15) -- (3.8,0);
    \node[draw=black,double,fill=white,fit=(1) (2) (3) (4) (5) (6) ,inner sep=1.3ex,circle] (cir) {};
  \end{pgfonlayer}
  
  \begin{pgfonlayer}{background}
  	\foreach \from in {21,...,29}
  	\foreach \to in {21,...,29} \ifthenelse{\from=\to}{}
    	{\path[->] (\from) edge [bend left] node{} (\to)};
    \end{pgfonlayer}
\end{tikzpicture}
	\caption{A compartmental spreading model with $|\comps| = 6$.  Dashed arrows denote external transitions, with their attractor groups circumscribed by dashed ellipses.  Solid arrows denote internal transitions.  For this particular example, the transitions $\{c \goesto f\}$ and $\{f \goesto d\}$ are external, with $\mathcal{A}(c,f) = \{f\}$ and $\mathcal{A}(f,d) = \{a,d\}$.}
\end{figure}
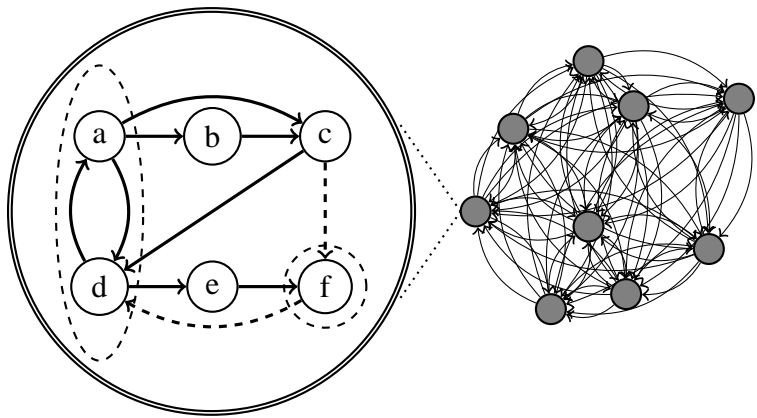

\subsection{Approximation Methods}

A common moment closure method in epidemic modeling is to replace the second-order moments $\mathbb{E} [\mathbbm{1}_{\{x_{i}(t) = c \}} \mathbbm{1}_{\{ x_{j}(t) = \tilde{c} \}}]$ by a product of expectations $\mathbb{E} [\mathbbm{1}_{\{x_{i}(t) = c \}}] \mathbb{E}[\mathbbm{1}_{\{ x_{j}(t) = \tilde{c} \}} ]$, in effect making an independence assumption.  While for particular model choices this treatment could be a good approximation (see, e.g., \cite{van2009virus,van2015accuracy} for discussion of the efficacy of this method as it pertains to the Susceptible-Infected-Susceptible  model), we notice two faults which exist when applied to general models.

Generally, this treatment can only give weak guarantees of accuracy, e.g. the existence of a neighborhood in which error is bounded \cite{singh2011approximate}, and cannot predict whether the system over-approximates or under-approximates the model.  Moreover, the systems constructed may over-approximate the model at some times, and under-approximate at others.
This hinders a designer's ability to relate the states of the approximated system to the moments of the modeled system with any confidence, preventing the development of rigorous guarantees with use of these approximations in many control applications, such as those studied in \cite{preciado2014optimal,CN-VMP-GJP:15-TCNS,watkins2015optimal}.  We note there are examples \cite{cator2012second, naasell2003extension} for which performing moment closure via an assumption of independence
have resulted in approximations which estimate impossible probabilities, e.g. values outside of the unit interval.  This, again, calls into question the efficacy of the technique; we will construct an alternative in what follows.

\subsection{Bounding Approximations} \label{sec:bounding}
To concentrate our analysis appropriately, we make a distinction between internal and external transitions.  Accordingly, we let $\mathcal{E}(c)$ denote the set of all compartments $c^\prime$ for which either or both of the transitions $\{c \rightarrow c^\prime \}$ and $\{c^\prime \goesto c\}$ are external.  We let $\mathcal{I}(c)$ denote the set of all compartments $c^\prime$ for which either or both of the transitions $\{c \rightarrow c^\prime \}$ and $\{c^\prime \goesto c\}$ are internal.  With this notation, we may rewrite the exact moment dynamics \eqref{eq:dynamics} as:
\begin{equation} \label{eq:exact_sep}
	\begin{aligned}
	&\frac{d \mathbb{E} \left[ \mathbbm {1}_{\{x_{i}(t) = c \}} \right]}{d t} =\\
	&\mathbb{E} \left[ \sum_{c^\prime \in \mathcal{E}(c)} \sum_{\tilde{c} \in \mathcal{A}(c^\prime, c)} \sum_{j \in V} \mathbbm{1}_{\{x_{i}(t) = c\}} \mathbbm{1}_{\{x_{j}(t) = \tilde{c} \}} \beta^{\{\tilde{c}; \, c^\prime \rightarrow c\}}_{ij}  \right] 
	- \mathbb{E} \left[ \sum_{c^\prime \in  \mathcal{E}(c)} \sum_{\tilde{c} \in \mathcal{A}(c, c^\prime)} \sum_{j \in V} \mathbbm{1}_{\{x_{i}(t) = c \}} \mathbbm{1}_{\{x_{j}(t) = \tilde{c} \}} \beta^{\{\tilde{c};\, c \rightarrow c^\prime\}}_{ij} \right]\\
	+ &\mathbb{E} \left[ \sum_{c^\prime \in  \mathcal{I}(c)}  \mathbbm{1}_{\{x_{i}(t) = c^\prime\}} \delta_i^{\{c^\prime \rightarrow c\}}  \right] 
		- \mathbb{E} \left[ \sum_{c^\prime \in  \mathcal{I}(c)} \mathbbm{1}_{\{x_{i}(t) = c \}} \delta_i^{\{c \rightarrow c^\prime \}} \right].
	\end{aligned}
\end{equation}

Note that we may evaluate the expectations and rewrite \eqref{eq:exact_sep} in terms of probabilities, which yields:
\begin{equation}
	\begin{aligned}
		&\frac{d \Pr(x_i = s)}{d t} =\\
		&\sum_{c^\prime \in  \mathcal{E}(c)} \sum_{\tilde{c} \in \mathcal{A}(c^\prime, c)} \sum_{j \in V} \Pr(x_i = c^\prime, x_j = c) \beta^{\{\tilde{c}; \, c^\prime \rightarrow c\}}_{ij}
		-\sum_{c^\prime \in  \mathcal{E}(c)} \sum_{\tilde{c} \in \mathcal{A}(c, c^\prime)} \sum_{j \in V}\Pr(x_i = c, x_j = c^\prime) \beta^{\{\tilde{c}; \; c \rightarrow c^\prime\}}_{ij} \\
		+&\sum_{c^\prime \in  \mathcal{I}(c)}  \Pr(x_{i}(t) = c^\prime) \delta_i^{\{c^\prime \rightarrow c\}}
			- \Pr(x_{i}(t) = c) \delta_i^{\{c \rightarrow c^\prime \}}.
		\end{aligned}
\end{equation}

This representation brings us close to a system we can approximate well: we need only to bound the second-order moment in terms of first-order moments.  For this, we recall that for any two events $A$ and $B$, the  Fr\'{e}chet inequalities (see,e.g., \cite{tankov2011improved}) give that $\Pr(A \cap B)$ satisfies
\begin{equation} \label{eq:frechet}
	\max\{0,\Pr(A)+ \Pr(B) - 1\} \leq \Pr(A \cap B) \leq \min\{\Pr(A),\Pr(B)\}.
\end{equation}

If we apply \eqref{eq:frechet} to \eqref{eq:exact_sep} appropriately and define the state $p_i^{\{c\}} = \Pr(x_i = \{c\})$, we can verify that the bounds:
\begin{equation} \label{eq:bound}
	\begin{aligned}
			\frac{d p_i^{\{s\}}}{d t} \geq&
			\sum_{c^\prime \in  \mathcal{E}(c)} \sum_{\tilde{c} \in \mathcal{A}(c^\prime, c)} \sum_{j \in V} \max\{0, p_i^{\{c^\prime\}}
			+ p_j^{\{c\}} - 1\} \beta^{\{\tilde{c}; \;c^\prime \rightarrow c\}}_{ij}
			- \sum_{c^\prime \in  \mathcal{E}(c)} \sum_{\tilde{c} \in \mathcal{A}(c, c^\prime)} \sum_{j \in V} \min\{p_i^{\{c\}}, p_j^{\{c^\prime\}}\} \beta^{\{\tilde{c}; \;c \rightarrow c^\prime\}}_{ij} \\
			&+\sum_{c^\prime \in  \mathcal{I}(c)}  p_i^{\{c^\prime \}} \delta_i^{\{c^\prime \rightarrow c\}}
				- p_i^{\{c\}} \delta_i^{\{c \rightarrow c^\prime \}}\\
%
		\frac{d p_i^{\{c\}}}{d t} \leq&
		\sum_{c^\prime \in  \mathcal{E}(c)} \sum_{\tilde{c} \in \mathcal{A}(c^\prime, c)} \sum_{j \in V} \min\{p_i^{\{c\}}, p_j^{\{c^\prime\}} \} \beta^{\{\tilde{c}; \;c^\prime \rightarrow c\}}_{ij}
		- \sum_{c^\prime \in  \mathcal{E}(c)} \sum_{\tilde{c} \in \mathcal{A}(c, c^\prime)} \sum_{j \in V} \max\{0,p_i^{\{c\}} + p_j^{\{c^\prime\}} - 1\} \beta^{\{\tilde{c}; \; c \rightarrow c^\prime\}}_{ij}\\
		&+\sum_{c^\prime \in  \mathcal{I}(c)}  p_i^{\{c^\prime\}} \delta_i^{\{c^\prime \rightarrow c\}} - p_i^{\{c\}} \delta_i^{\{c \rightarrow c^\prime \}}
	\end{aligned}
\end{equation}
hold.  We will now show how such bounds can be used to approximate the compartment membership probabilities of the general spreading model.

\begin{lem}[Multivariate Comparison] \label{lem:comparison}
	Consider a system of equations
	$$\dot{x} = f(x)$$
	with $x \in D \subset \real^n$ and $f : D \rightarrow \real^n$.
	If there exist vector functions $\bar{f}$ and $\ubar{f}$ defined on $D$ such that for each component $\bar{f}_i$ and $\ubar{f}_i$ we have that
	\begin{align*}
	f_i(z) &\leq \bar{f}_i(\bar{z},\ubar{z}), \\
	f_i(z) &\geq \ubar{f}_i(\bar{z},\ubar{z}) ,
	\end{align*}
	hold wherever $\ubar{z}_i = z_i = \bar{z}_i$, $\ubar{z}_j \leq z_j \leq \bar{z}_j$ for all $j \neq i$, and 
	$\ubar{z}, z, \bar{z} \in D$.
	Then, the solutions to the system
	\begin{equation*}
		\begin{aligned}
		\dot{\bar{x}}= \bar{f}(\bar{x},\ubar{x}),\\
		\dot{\ubar{x}} = \ubar{f}(\bar{x},\ubar{x}),
		\end{aligned}
	\end{equation*}
	with initial conditions $\bar{x}(t_0) = \ubar{x}(t_0) = x(t_0)$, satisfy
	$\ubar{x}_i(t) \leq x_i(t) \leq \bar{x}_i(t)$ for all $i$ and $t \geq t_0$, 
	provided that the solutions $\ubar{x}(t)$, $x(t)$, and $\bar{x}(t)$ are unique and continuous. 
\end{lem}

\begin{proof}
	We argue by contradiction.  In particular, we will show that no $\tau \in [t_0, \infty)$ exists such that $\ubar{x}_i(\tau) \leq x_i(\tau) \leq \bar{x}_i(\tau)$ fails to hold.  Assume for contradiction that $\tau^\prime$ is first such time that the chain of inequalities fails to hold.  We may assume $\bar{x}_i(\tau^\prime) < x_i(\tau^\prime)$  for some $i$; the opposite case will hold by a similar argument.  By continuity of $\bar{x}_i(t)$ and $x_i(t)$, there must exist some $\epsilon > 0$ such that $\bar{x}_i(\tau^\prime - \epsilon) = x_i(\tau^\prime - \epsilon)$.  Let $\epsilon^\prime$ be the smallest such value.  Then, we must have that $\dot{\bar{x}}_i(\tau^\prime - \epsilon^\prime) < \dot{x}_i(\tau^\prime - \epsilon^\prime)$.  However, this contradicts our hypothesis: at $t = \tau^\prime - \epsilon^\prime$, it must be that $\dot{\ubar{x}}_i = \ubar{f}_i(\bar{x},\ubar{x}) \leq f_i(x) \leq \bar{f}_i(\bar{x},\ubar{x}) = \dot{\bar{x}}_i$ holds.  Hence, no such $\tau^\prime$ exists, which implies $\ubar{x}_i(t) \leq x_i(t) \leq \bar{x}_i$ for all $i$ and all $t \geq t_0$.
	
	Note that the case in which $\bar{x}_i(\tau^\prime) < x_i(\tau^\prime)$  for countably many $i$ follows almost immediately.  Formally, let $\mathcal{I}$ be the set of all $i$ for which $\bar{x}_i(\tau^\prime) < x_i(\tau^\prime)$ holds.  Then, by continuity of $x_i(t)$, there exist $\epsilon_i > 0$ such that $\bar{x}_i(\tau^\prime - \epsilon_i) = x_i(\tau^\prime - \epsilon_i)$ holds.  For each $i$, let $\epsilon_i^\prime$ be the smallest value, and define the index set $\mathcal{J}$ to be a permutation of $\mathcal{I}$ such that $\epsilon_{j_1}^\prime \geq \epsilon_{j_2}^\prime \geq  \dots \geq \epsilon_{j_n}^\prime$.  Then, we must have that $\dot{\bar{x}}_{j_1}(\tau^\prime - \epsilon_{j_1}^\prime) < \dot{x}_{j_1}(\tau^\prime - \epsilon_{j_1}^\prime)$, but this breaks our hypothesis.  The proof is completed by noting that the contradiction established carries through by induction.
\end{proof}

Before continuing, it is useful to note that the functions $\bar{f}$ and $\ubar{f}$ required by the hypothesis of Lemma \ref{lem:comparison} are \emph{not} in general unique.  However, it is trivially true that the pointwise minimum of all upper-bounding trajectories is itself an upper-bounding trajectory, and the pointwise maximum of all lower-bounding trajectories is itself is itself a lower-bounding trajectory: having multiple bounding systems can only help our analysis for any particular model.
 

It is easy to see that this result gives us a mechanism for constructing a system which bounds the evolution of $p_i^{\{s\}}(t)$ by leveraging the bound \eqref{eq:bound}.
\begin{theorem} \label{thm:bounding_sys}
	Given any compartmental spreading process governed by the dynamics \eqref{eq:exact_sep}, the solutions of the system
	\begin{equation} \label{eq:bounding}
		\begin{aligned}
		\dot{\bar{p}}_i^{\{c\}}
		= &\sum_{s^\prime \in  \mathcal{E}(c)} \sum_{\tilde{c} \in  \mathcal{A}(c^\prime, c)} \sum_{j \in V} \min\{(1-\bar{p}_i^{ \{c\}}), \bar{p}_j^{\{c\}} \}  \beta^{\{\tilde{c}; \; c^\prime \rightarrow c\}}_{ij}
				- \sum_{c^\prime \in  \mathcal{E}(c)} \sum_{\tilde{c} \in  \mathcal{A}(c, c^\prime)} \sum_{j \in V} \max\{0,\bar{p}_i^{\{c\}} + \ubar{p}_j^{\{\tilde{c}\}} - 1\} \beta^{\{\tilde{c}; \;c \rightarrow c^\prime\}}_{ij} \\
				&+\sum_{c^\prime \in  \mathcal{I}(c)} (1 - \bar{p}_i^{\{c\}}) \delta_i^{\{c^\prime \rightarrow c\}}
				- \bar{p}_i^{\{c\}} \delta_i^{\{c \rightarrow c^\prime \}}\\
		\dot{\ubar{p}}_i^{\{c\}}
		=&\sum_{c^\prime \in  \mathcal{E}(c)} \sum_{\tilde{c} \in  \mathcal{A}(c^\prime, c)} \sum_{j \in V} \max\{0, \ubar{p}_i^{\{\tilde{c}\}} + \ubar{p}_j^{\{c\}} - 1\} \beta^{\{\tilde{c}; \; c^\prime \rightarrow c\}}_{ij} - \sum_{c^\prime \in  \mathcal{E}(c)} \sum_{\tilde{c} \in  \mathcal{A}(c, c^\prime)} \sum_{j \in V}\min\{\ubar{p}_i^{\{c\}}, \bar{p}_j^{\{\tilde{c}\}}\} \beta^{\{\tilde{c}; \; c \rightarrow c^\prime\}}_{ij} \\
		+&\sum_{c^\prime \in  \mathcal{I}(c)}  \ubar{p}_i^{\{c^\prime\}} \delta_i^{\{c^\prime \rightarrow c\}}
		- \ubar{p}_i^{\{c\}} \delta_i^{\{c \rightarrow c^\prime \}}\\
		\end{aligned}
	\end{equation}
	satisfy $0 \leq \ubar{p}_i^{\{c\}}(t) \leq p_i^{\{c\}} (t) \leq \bar{p}_i^{\{c\}} (t) \leq 1$ for all $t \geq t_0$. for all $\ubar{p}_i^{\{c\}}(t_0) = p_i^{\{c\}}(t_0) = \bar{p}_i^{\{c\}}(t_0) \in [0,1]$.
\end{theorem}
\begin{proof}
	We begin by noting that the dynamics $\dot{\ubar{p}}_i$ and $\dot{\bar{p}}_i$ (and hence $\dot{p}_i$) are globally Lipschitz with constant $$L = \max_{c} \{|\sum_{c^\prime \in \mathcal{E}(c)} \sum_{\tilde{c} \in  \mathcal{A}(c^\prime, c)} \sum_{j \in V} \beta_{ij}^{\{\tilde{c}; \; c^\prime \goesto c\}} + \sum_{c^\prime \in \mathcal{I}(c)} \delta_i^{\{c^\prime \goesto c\}}|, |\sum_{c^\prime \in \mathcal{E}(c)} \sum_{\tilde{c} \in  \mathcal{A}(c, c^\prime)} \sum_{j \in V} \beta_{ij}^{\{\tilde{c}; \;c \goesto c^\prime\}} + \sum_{c^\prime \in \mathcal{
	I}(c)} \delta_i^{\{c \goesto c^\prime\}}|\}.$$  Hence, the solutions $\ubar{p}_i^{\{c\}}(t)$, $p_i^{\{c\}}(t)$ and $\bar{p}_i^{\{c\}}(t)$ exist, are unique, and are continuous.  Now, we note $p_i^{\{c^\prime\}} = (1 - \sum_{\hat{c} \neq c^\prime} p_i^{\{\hat{c}\}}) \leq (1 - p_i^{\{c\}})$ for all $i$, $c$ and $c^\prime$.  It then follows that the inequalities $\ubar{p}_i(t) \leq p_i(t) \leq \bar{p}_i(t)$ hold for all $t \geq t_0$ by application of Theorem \ref{lem:comparison}.  
	
	To show that $\bar{p}_i^{\{c\}} (t) \leq 1$, we note that by construction $\dot{\bar{p}}^{\{c\}} \leq 0$ whenever $\bar{p}_i^{\{c\}} = 1$, and by hypothesis $\bar{p}_i^{\{c\}} (t_0) \in [0,1]$.  So, by continuity of $\bar{p}_i^{\{c\}}(t)$, we have $\bar{p}_i^{\{c\}}(t) \leq 1$ for all $t \geq t_0$.  To show that $0 \leq \ubar{p}_i^{\{c\}}(t)$ for all $t \geq t_0$, we note that when $\ubar{p}_i^{\{s\}} = 0$, we have $\dot{\ubar{p}}_i^{\{c\}} \geq 0$.  So, by $\ubar{p}_i^{\{c\}} \in [0,1]$ and continuity of $\ubar{p}_i^{\{c\}}(t)$, we have $ 0 \leq \ubar{p}_i^{\{c\}}(t)$ for all $t \geq t_0$.
	Taken together, we have $0 \leq \ubar{p}_i(t) \leq p_i(t) \leq \bar{p}_i (t) \leq 1$ for all $t \geq t_0$.
\end{proof}

\begin{remark}[Existence of Alternate Bounding Systems] \label{rem:alternate}
{\rm	
Note that this construction is just one possible system satisfying the hypothesis of Theorem \ref{lem:comparison}.  In particular, with further knowledge of the model's structure we may construct alternate bounding systems, some of which may be more suitable to the problem at hand (e.g. easier to analyze or more accurate) than the one provided in Theorem \ref{thm:bounding_sys}.  As particular examples, we demonstrate in Section \ref{subsec:SIR} that a better approximation system can be constructed for the Susceptible-Infected-Removed ($SIR$) model.  We show the same for the Susceptible-Infected-Susceptible ($SIS$) model in Section \ref{subsec:SIS}.  Our construction testifies that one such system is possible for \emph{any} compartmental spreading process, but is in no means a guarantee that it is the \emph{best} approximation. \oprocend
}
\end{remark}
 
The reader may note that for all spreading process, all compartment membership probabilities at each node must sum to 1, and that the trajectories of \eqref{eq:bounding} need not guarantee this of all possible evolutions allowed by the bounding trajectories.  Hence, it is worthwhile to consider a means for eliminating any impossible state trajectories, as this may improve the approximations. 

\begin{cor}[Eliminating Impossible Trajectories] \label{cor:eliminating}
Consider the set of trajectories $\{\ubar{p}_i^{\{c\}},\bar{p}_i^{\{c\}}\}_{\{c \in \comps, i \in V\}}$ of \eqref{eq:bounding}.  Then, for each $c$ and $i$, the trajectory $\max\{0,1 - \sum_{c^\prime \in \comps \setminus \{c\}} \bar{p}_i^{\{c\}} \}$ is an under-approximator of $p_i^{\{c\}}$ and the trajectory $\min \{1, 1 - \sum_{c^\prime \in \comps \setminus \{c\}} \ubar{p}_i^{\{c\}}\}$ is an over-approximator of $p_i^{\{c\}}$.
\end{cor}

\begin{proof}
	This follows immediately from noting $\ubar{p}_i^{\{c\}} \leq p_i^{\{c\}} \leq \bar{p}_i^{\{c\}}$ holds for all $i$ and $c$.
\end{proof}

We will see the effects of Corollary \ref{cor:eliminating} when investigating the efficacy of the bounding system \eqref{eq:bounding} in Section \ref{sec:app}.

\subsection{Leveraging Correlations}
When considering a process for which it is known that the involved variables are non-negatively correlated (e.g. $SIS$ and $SIR$ \cite{cator2014nodal}), we can improve upon the upper-bounding trajectory.  To show this formally, we will use the metric $d_{\{t_0,T\}}(f,g) = \max_{t \in [t_0,T]} |f(t) - g(t)|$ (where $f$ and $g$ are assumed continuous) as a measure for the quality of the bounding system - it tells us the size of the largest gap which occurs between two trajectories over the initial interval $[t_0,T]$.  Heuristically, this gives us a measure of the error incurred in the approximation.  We can state a sufficient condition for when one bounding system is better than another, with respect to $d_{\{t_0,T\}}$.

\begin{cor}[Tighter Bounding Systems] \label{cor:tight}
Consider the metric $d_{\{t_0,T\}}(f,g) = \max_{t \in [t_0,T]} |f(t) - g(t)|$.  Then, a system
\begin{equation*}
	\begin{aligned}
	\dot{\tilde{x}} = \tilde{f}(\tilde{x},\utilde{x})\\
	\dot{\utilde{x}} = \utilde{f}(\tilde{x},\utilde{x})\\
	\dot{\bar{x}} = \bar{f}(\bar{x},\ubar{x})\\
	\dot{\ubar{x}} = \ubar{f}(\bar{x},\ubar{x})
	\end{aligned}
\end{equation*}
which satisfies $\dom(\tilde{f}) = \dom(\bar{f})$ and $\ubar{f}_i(\bar{x},\ubar{x}) \leq \utilde{f}_i(\tilde{x},\utilde{x}) \leq \tilde{f}_i(\tilde{x},\utilde{x}) \leq \bar{f}_i(\bar{x},\ubar{x})$ for all $(\tilde{x},\utilde{x}) \in \dom(\tilde{f})$ and all $(\bar{x},\ubar{x}) \in \dom(\bar{f})$ engenders solutions $\{\utilde{x}_i(t),\tilde{x}_i(t),\ubar{x}_i(t),\bar{x}_i(t)\}_{i = 1}^n$ which satisfy $d_{\{t_0,T\}}(\utilde{x}_i,\tilde{x}_i) \leq d_{\{t_0,T\}}(\ubar{x}_i,\bar{x}_i)$. for all $i$, provided the system has unique, continuous solutions,
and $\tilde{x}_i(t_0) = \utilde{x}_i(t_0) = \bar{x}_i(t_0) = \ubar{x}_i(t_0)$.
\end{cor}

 \begin{proof}
 	Application of Theorem \ref{thm:bounding_sys} gives us that $\ubar{x}_i(t) \leq \utilde{x}_i(t) \leq \tilde{x}_i(t) \leq \bar{x}_i(t)$ holds for all $t \geq t_0$, and so the result follows immediately.
 \end{proof}
 
 We may use this result to examine the effect of using knowledge of non-negative or non-positive correlation in designing a better approximating system than that which is constructed in Theorem \ref{thm:bounding_sys}.
 
 \begin{theorem} \label{thm:alt_sys}
 	Recall that the covariance of two random variables $X$ and $Y$ is defined $\sigma(X,Y) = \mathbb{E}[XY] - \mathbb{E}[X] \mathbb{E}[Y]$.  Now, suppose that for all $i, j \in [n]$ and $c, c^\prime \in \comps$, we have $\sigma(\mathbbm{1}_{\{x_i(t) = c \}},\mathbbm{1}_{\{x_j(t) = c^\prime \}}) \geq 0$.  Then, 
 	\begin{equation} \label{ineq:nonneg_corr}
 		\Pr(x_i(t) = c) \Pr(x_j(t) = c^\prime) \leq \min\{\Pr(x_i(t) = c), \Pr(x_j(t) = c^\prime)\} -
 		\sigma(\mathbbm{1}_{\{x_i(t) = c \}}, \mathbbm{1}_{\{x_j(t) = c^\prime \}}),
 	\end{equation}
 	holds, and the bounding system given by
 	\begin{equation} \label{eq:corr_sys}
 		\begin{aligned}
 			\dot{\tilde{p}}_i^{\{s\}}
 			= &\sum_{c^\prime \in  \mathcal{E}(c)} \sum_{\tilde{c} \in  \mathcal{A}(c^\prime, c)} \sum_{j \in V} (1-\tilde{p}_i^{ \{c\}}) (\tilde{p}_j^{\{c\}})  \beta^{\{\tilde{c}; \; c^\prime \rightarrow c\}}_{ij}
 			- \sum_{c^\prime \in  \mathcal{E}(c)} \sum_{\tilde{c} \in  \mathcal{A}(c, c^\prime)} \sum_{j \in V} \max\{0,\tilde{p}_i^{\{c\}} + \utilde{p}_j^{\{\tilde{c}\}} - 1\} \beta^{\{\tilde{c}; \;c \rightarrow c^\prime\}}_{ij} \\
 			&+\sum_{c^\prime \in  \mathcal{I}(c)} (1 - \tilde{p}_i^{\{c\}}) \delta_i^{\{c^\prime \rightarrow c\}}
 			- \tilde{p}_i^{\{c\}} \delta_i^{\{c \rightarrow c^\prime \}}\\
 			\dot{\utilde{p}}_i^{\{c\}}
 			=&\sum_{c^\prime \in  \mathcal{E}(c)} \sum_{\tilde{c} \in  \mathcal{A}(c^\prime, c)} \sum_{j \in V} \max\{0, \utilde{p}_i^{\{\tilde{c}\}} + \utilde{p}_j^{\{c\}} - 1\} \beta^{\{\tilde{c}; \; c^\prime \rightarrow c\}}_{ij} - \sum_{c^\prime \in  \mathcal{E}(c)} \sum_{\tilde{c} \in  \mathcal{A}(c, c^\prime)} \sum_{j \in V}(\utilde{p}_i^{\{c\}})(\tilde{p}_j^{\{\tilde{c}\}}) \beta^{\{\tilde{c}; \; c \rightarrow c^\prime\}}_{ij} \\
 			+&\sum_{c^\prime \in  \mathcal{I}(c)}  \utilde{p}_i^{\{c^\prime\}} \delta_i^{\{c^\prime \rightarrow c\}}
 			- \utilde{p}_i^{\{c\}} \delta_i^{\{c \rightarrow c^\prime \}}\\
 		\end{aligned}
 	\end{equation}
 	engenders solutions which satisfy $d_{\{t_0,T\}}(\tilde{p}_i^{\{c\}},\utilde{p}_i^{\{c\}}) \leq d_{\{t_0,T\}}(\bar{p}_i^{\{c\}},\ubar{p}_i^{\{c\}})$ for all $i \in [n]$, $c \in \comps$ and all choices of $T \geq t_0$, where $\bar{p}_i^{\{c\}}$ and $\ubar{p}_i^{\{c\}}$ are the solutions of the bounding system given by \eqref{eq:bounding}.
 \end{theorem}
 
 \begin{proof}
 	The inequality \eqref{ineq:nonneg_corr} follows from the Fr\'{e}chet inequalities and the definition of covariance.  Application of \eqref{ineq:nonneg_corr} shows that \eqref{eq:corr_sys} satisfies the hypothesis of Theorem \ref{cor:tight}, and so the claim follows immediately.
 \end{proof}
 
 \begin{remark}[Correlations Between Some Variables] \label{rem:correlations}
 {\rm
 	It is easy to show that a similar result holds when not all pairs of variables are non-negatively correlated: simply leave the Fr\'{e}chet bound substitutions intact wherever necessary.  A similar result holds for pairs of non-positively correlated variables, where the product is substituted for the appropriate Fr\'{e}chet bound.  Moreover, \emph{any} \emph{a priori} knowledge of non-negative (or non-positive) correlation allows for the construction of an approximating system which is better with respect to $d_{\{t_0,T\}}$.\oprocend
 	}
 \end{remark}
 
 Theorem \ref{thm:alt_sys} gives us a rigorous condition under which better bounding systems can be constructed.  Whenever we have \emph{a priori} knowledge of the correlation of variable pairs, we may use it to construct a better system.  In fact, application of the Fr\'{e}chet inequalities are, in effect, making an assumption that the correlation between variable pairs are \emph{as bad as possible}, and so avoid the need for \emph{a priori} knowledge.  We may formalize this by noting that the upper bound is tight whenever the variable pair realizes maximal positive covariance (i.e. $\sigma(\mathbbm{1}_{\{x_i(t) = c\}}, \mathbbm{1}_{\{x_j(t) = c^\prime\}}) = 1$ for all $t$), and the lower bound is tight whenever the variable pair realizes maximal negative covariance (i.e. $\sigma(\mathbbm{1}_{\{x_i(t) = c\}}, \mathbbm{1}_{\{x_j(t) = c^\prime\}}) = -1$ for all $t$).

\section{Example applications} \label{sec:app}
Having established our main technical contributions, we now demonstrate the utility of the developed results by applying them to several representative spreading processes.  We begin by demonstrating that no approximation is necessary for fully internal processes (Section \ref{subsec:internal}), then continue by studying bounding approximations to Susceptible-Infected-Removed ($SIR$; Section \ref{subsec:SIR}), Susceptible-Infected-Susceptible ($SIS$; Section \ref{subsec:SIS}), Susceptible-Infected 1-Susceptible-Infected 2-Susceptible ($S I_1 S I_2 S$; Section \ref{subsec:SISIS}), and Susceptible-Exposed-Infected-Vigilant ($SEIV$; Section \ref{subsec:SEIV}) processes.

\subsection{Fully internal processes} \label{subsec:internal}


As a first example,we will consider a process in which all transitions are internal.  We note that that the dynamic bounds \eqref{eq:bound} which describe the system reduce to 
\begin{equation}
	\begin{aligned}
		&\frac{d p_i^{\{c\}}}{d t} \geq
		&\sum_{c^\prime \in  \mathcal{I}(c)}  p_i^{\{c^\prime\}} \delta_i^{\{c^\prime \rightarrow c\}}
		- \sum_{c^\prime \in  \mathcal{I}(c)} p_i^{\{c\}} \delta_i^{\{c \rightarrow c^\prime \}}\\				
		&\frac{d p_i^{\{c\}}}{d t} \leq
		&\sum_{c^\prime \in  \mathcal{I}(c)}  p_i^{\{c^\prime\}} \delta_i^{\{c^\prime \rightarrow c\}}
		- \sum_{c^\prime \in  \mathcal{I}(c)} p_i^{\{c\}} \delta_i^{\{c \rightarrow c^\prime \}}
	\end{aligned}
\end{equation}

Hence, we may consider the exact dynamics
\begin{equation}
	\dot{p}_i^{\{c\}} = \sum_{c^\prime \in  \mathcal{I}(c)}  p_i^{\{c^\prime\}} \delta_i^{\{c^\prime \rightarrow c\}}
			- \sum_{c^\prime \in  \mathcal{I}(c)} p_i^{\{c\}} \delta_i^{\{c \rightarrow c^\prime \}}
\end{equation}
directly, with no need for approximation.  This demonstrates conclusively that the complexity of spreading processes arises from the coupling effect of external transitions - our method of approximation leads back to an exact representation in this case.  This is a worthwhile, if obvious, result.

\subsection{$SIR$} \label{subsec:SIR}
The Susceptible-Infected-Removed (pictured in Figure \ref{fig:SIR}; see, e.g., \cite{keeling2005networks}) epidemic process has $|\comps| = 3$, where the three possible compartments an agent may take are `Susceptible,' ($S$), `Infected,' ($I$), and `Removed,' ($R$).  We allow the transitions $\{S \goesto I\}$ to be external with rate $\lambda_i^{\{S \goesto I\}}(t) = \sum_{j \in V} \mathbbm{1}_{\{x_j(t) = I\}} \beta_{ij}^{\{I; \; S \goesto I\}}$ and the transitions $\{I \goesto R\}$ to occur with rate $\delta_i^{\{I \goesto R\}}$.  We do not allow any other transitions.  In particular, once an agent has attained state $R$, it will remain there \emph{ad infinitum} - this allows for the modeling of lasting immunity after initial infection.

It is easy to see that the bounding system \eqref{eq:bounding} applies here, but we may apply Theorem \ref{thm:alt_sys} and the fact that $\mathbbm{1}_{\{x_i(t) = S\}}$ and  $\mathbbm{1}_{\{x_j(t) = I\}}$ are non-negatively correlated \cite{cator2014nodal} to construct the \emph{ad hoc} approximating system:
\begin{equation} \label{eq:SIR_dyn}
	\begin{aligned}
	\dot{\bar{p}}_i^{\{S\}} &= - \sum_{j \in V} \max\{0,\bar{p}_i^{\{S\}} + \ubar{p}_j^{\{I\}} - 1\} \beta_{ij}^{\{I; \, S \goesto I\}}\\
	\dot{\ubar{p}}_i^{\{S\}} &= - \sum_{j \in V} \ubar{p}_i^{\{S\}} \bar{p}_j^{\{I\}} \beta_{ij}^{\{I; \, S \goesto I\}}\\
	\dot{\bar{p}}_i^{\{I\}} &= \sum_{j \in V} (1 - \bar{p}_i^{\{I\}}) \bar{p}_j^{\{I\}} \beta_{ij}^{\{I; \, S \goesto I\}} - \bar{p}_i^{\{I\}} \delta_{i}^{\{I \goesto R\}} \\
	\dot{\ubar{p}}_i^{\{I\}} &= \sum_{j \in V} \max\{0,\ubar{p}_i^{\{S\}} + \ubar{p}_j^{\{I\}} - 1\} \beta_{ij}^{\{I; \, S \goesto I\}} - \ubar{p}_i^{\{I\}} \delta_{i}^{\{I \goesto R\}} \\
	\dot{\bar{p}}_i^{\{R\}} &= (1 - \bar{p}_i^{\{R\}}) \delta_{i}^{\{I \goesto R\}}\\
	\dot{\ubar{p}}_i^{\{R\}} &= \ubar{p}_i^{\{I\}} \delta_{i}^{\{I \goesto R\}}
	\end{aligned}
\end{equation}

which, by Corollary \ref{cor:tight}, must be as good an approximation as \eqref{eq:bounding} with respect to $d_{\{t_0,T\}}$, and in simulation appears to be better.

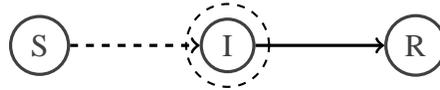
\begin{figure}[h]
\centering
	\begin{tikzpicture}[->, thick]
		\begin{scope}[shift={(1,1)}]			
		\node[circle,very thick,draw=black,opacity=0.75] (S) at (0,0) {S};
		\node[circle,very thick,draw=black,opacity=0.75] (I) at (2.5,0) {I};
		\node[circle,very thick,draw=black,opacity=0.75] (R) at (5,0) {R};
		\path [->,very thick, dashed] (S) edge node[xshift = 0.1cm, yshift = 0.3cm, dashed] {} (I);
		\node[draw=black,dashed,fit=(I) ,inner sep=0.1ex,ellipse] (Icirc) {};
		\path [->,very thick] (I) edge node[xshift = 0.3cm, yshift = 0.3cm] {} (R);
		\end{scope}
	\end{tikzpicture}
	\caption{The compartmental diagram of $SIR$.  The dashed lines indicate external transitions; the solid lines indicate internal transitions.} \label{fig:SIR}
\end{figure}

{
	\psfrag{0}[cc][cc]{\tiny 0}
		\psfrag{0.005}[cc][cc]{\tiny 5}
		\psfrag{0.05}[cc][cc]{\tiny 100}
		\psfrag{.01}[cc][cc]{\tiny 150}
		\psfrag{200}[cc][cc]{\tiny 200}
		\psfrag{250}[cc][cc]{\tiny 250}
		\psfrag{300}[cc][cc]{\tiny 300}
		\psfrag{300}[cc][cc]{\tiny 20}
		\psfrag{300}[cc][cc]{\tiny 40}
		\psfrag{300}[cc][cc]{\tiny 60}
		\psfrag{300}[cc][cc]{\tiny 80}
		\psfrag{300}[cc][cc]{\tiny 100}
		\psfrag{300}[cc][cc]{\tiny 120}
\begin{figure}[h]
	\centering
	\begin{subfigure}[b]{0.5\textwidth}
		\includegraphics[width =\textwidth]{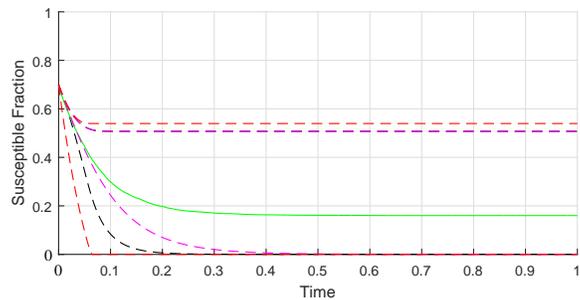}
		\caption{Expected Susceptible Fraction}
	\end{subfigure}
	~
	\begin{subfigure}[b]{0.5\textwidth}
		\includegraphics[width =\textwidth]{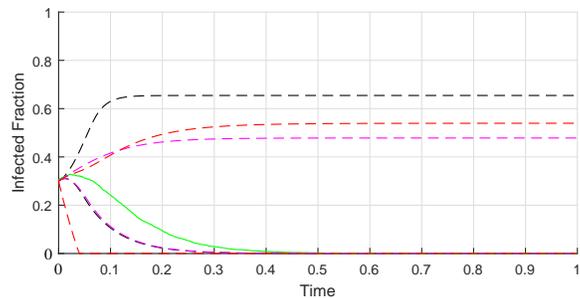}
		\caption{Expected Infected Fraction}
	\end{subfigure}
	~
	\begin{subfigure}[b]{0.5\textwidth}
		\includegraphics[width =\textwidth]{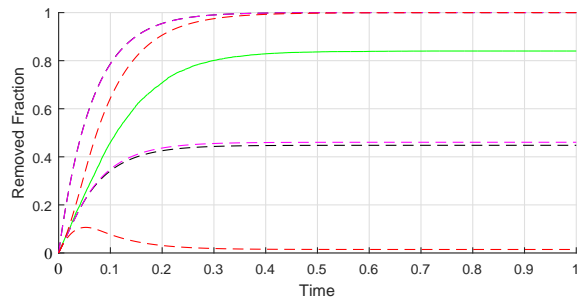}
		\caption{Expected Removed Fraction}
	\end{subfigure}
	\caption{Simulated Dynamics for the SIR Process on a 100-node random graph with connection probability p = 0.2.  The black dashed lines are the trajectories of the bounding system \eqref{eq:bounding}; the red dashed lines are the result of applying Corollary \ref{cor:eliminating} to the trajectories of \eqref{eq:bounding}; the magenta dashed lines are the trajectories of \eqref{eq:SIR_dyn};the green solid line is generated from a 100-trial Monte Carlo simulation of the process.  Note that all trajectories given are the average compartmental membership probability approximations for the graph.} \label{fig:SIR_sim}
\end{figure}}


We note the results of a Monte Carlo simulation of the exact process as compared to the trajectories of the bounding system, which are given in Figure \ref{fig:SIR_sim}.  From this we can see that the modeled process dynamics behave expected: the crude bounding system \eqref{eq:bounding} gives the loosest approximation, the system leveraging non-negative correlation \eqref{eq:SIR_dyn} gives a better approximation, and the expected trajectory lies neatly between the over- and under-approximating trajectories of the bounding system.  We note also that the trajectory of the original process is ``tight,'' to the upper and lower bounding trajectories in some domain for each process compartment.  This is perhaps an indication that the studied bounding systems are close to as good as one can expect for this process, though making a formal claim to this end is difficult, and so this observation must be treated as a heuristic claim.  In particular, it is in general difficult to predict \emph{a priori} for which values in the domain a particular bounding trajectory will be tight.

\subsection{$SIS$} \label{subsec:SIS}
The Susceptible-Infected-Susceptible ($SIS$) process (pictured in Figure \ref{fig:SIS}; see, e.g., \cite{van2009virus}) epidemic process has $|\comps| = 2$, where the two possible compartments an agent may take are `Susceptible,' ($S$) and `Infected,' ($I$).  We allow the transitions $\{S \goesto I\}$ to be external with rate $\lambda_{i \, (t)}^{\{S \goesto I\}} = \sum_{j \in V} \mathbbm{1}_{\{x_j(t) = I\}} \beta_{ij}^{\{I; \; S \goesto I\}}$ and the transitions $\{I \goesto S\}$ to occur with rate $\delta_i^{\{I \goesto S\}}$.  This allows for the modeling of diseases in which infection may be \emph{recurrent}, i.e. there is no mechanism for immunity (or removal) as in $SIR$.

Some recent literature (see, e.g., \cite{preciado2014optimal}) has been concerned with controlling the $SIS$ process via its mean-field approximation.  Though an argument can be made from the combined works of \cite{van2009virus} and \cite{cator2014nodal} that the mean-field approximation of $SIS$ is an upper-bound on the first moment of the process, we can use our bounding system framework to study the interrelation of the bounding system method developed here and the standard mean-field approximation, as they relate to control of the process.

It is easy to note that the bounding system \eqref{eq:bounding} applies here.  However, we may leverage the non-negative correlation of $\mathbbm{1}_{\{x_i(t) = I\}}$ and $\mathbbm{1}_{\{x_j(t) = I\}}$ for all $(i,j) \in V \times V$ established in \cite{cator2014nodal} to write three \emph{ad hoc} approximating systems for the $SIS$ process:
\begin{equation} \label{eq:SIS_dyn}
	\begin{cases}
		\dot{\bar{p}}_i^{\{I\}} &= \sum_{j \in V} \beta_{ij}^{\{I; \; S \goesto I\}} (1 - \bar{p}_i^{\{I\}})\bar{p}_j^{\{I\}} - \bar{p}_i^{\{I\}} \delta_i^{\{I \goesto S\}}\\
		\dot{\ubar{p}}_i^{\{I\}} &= \sum_{j \in V} \beta_{ij}^{\{I; \; S \goesto I\}} \max\{0,\ubar{p}_i^{\{S\}} + \ubar{p}_j^{\{I\}} - 1\} - \ubar{p}_i^{\{I\}} \delta_i^{\{I \goesto S\}}\\
		\dot{\bar{p}}_i^{\{S\}} &= -\sum_{j \in V} \beta_{ij}^{\{I; \; S \goesto I\}} \max\{0,\bar{p}_i^{\{S\}} + \ubar{p}_j^{\{I\}} - 1\} + (1-\bar{p}_i^{\{S\}}) \delta_i^{\{I \goesto S\}}\\
		\dot{\ubar{p}}_i^{\{S\}} &= -\sum_{j \in V} \beta_{ij}^{\{I; \; S \goesto I\}} \ubar{p}_i^{\{S\}} \bar{p}_j^{\{I\}} + (1-\bar{p}_i^{\{S\}}) \delta_i^{\{I \goesto S\}},
	\end{cases}
\end{equation}

\begin{equation} \label{eq:SIS_2}
	\begin{cases}
		\dot{\bar{p}}_i^{\{I\}} &= \sum_{j \in V} \beta_{ij}^{\{I; \; S \goesto I\}} (1 - \bar{p}_i^{\{I\}})\bar{p}_j^{\{I\}} - \bar{p}_i^{\{I\}} \delta_i^{\{I \goesto S\}}\\
		\dot{\ubar{p}}_i^{\{I\}} &= \sum_{j \in V} \beta_{ij}^{\{I; \; S \goesto I\}} \max\{0,\ubar{p}_j^{\{I\}} - \ubar{p}_i^{\{I\}}\} - \ubar{p}_i^{\{I\}} \delta_i^{\{I \goesto S\}}\\
		\bar{p}_i^{\{S\}} &= 1 - \ubar{p}_i^{\{I\}}\\
		\ubar{p}_i^{\{S\}} &= 1 - \bar{p}_i^{\{I\}},
	\end{cases} \hspace*{11.5ex}
\end{equation}
and 
\begin{equation} \label{eq:SIS_3}
	 \begin{cases}
		\dot{\bar{p}}_i^{\{S\}} &= -\sum_{j \in V} \beta_{ij}^{\{I; \; S \goesto I\}} \max\{0,\bar{p}_i^{\{S\}} - \bar{p}_j^{\{S\}} \} + (1-\bar{p}_i^{\{S\}}) \delta_i^{\{I \goesto S\}}\\
		\dot{\ubar{p}}_i^{\{S\}} &= -\sum_{j \in V} \beta_{ij}^{\{I; \; S \goesto I\}} \ubar{p}_i^{\{S\}} (1 - \ubar{p}_j^{\{S\}}) + (1-\bar{p}_i^{\{S\}}) \delta_i^{\{I \goesto S\}}\\
		\bar{p}_i^{\{I\}} &= 1 - \ubar{p}_i^{\{S\}}\\
		\ubar{p}_i^{\{I\}} &= 1 - \bar{p}_i^{\{S\}}.
	\end{cases} \hspace*{3ex}
\end{equation}

Note the equations which govern the dynamics of $\bar{p}_i^{\{I\}}$ in \eqref{eq:SIS_dyn} and \eqref{eq:SIS_2} are identical to those of the standard mean-field approximation \cite{van2009virus}.  This is an important fact.  Since the mean-field trajectories are trajectories of a bounding system, the work done in control of $SIS$ via mean-field approximation (see, e.g. \cite{preciado2014optimal}) is valid: controlling the mean-field approximation to a disease-free state must also control the first moment of the process to $0$.  However, taken together, the states of \eqref{eq:SIS_dyn} and \eqref{eq:SIS_2} provide a complete picture of the system's expected evolution, with upper and lower bounds on the probability of membership of both compartments.  This is an advantage of this approach over the standard mean-field approximation, and one which has not yet been studied in the literature.

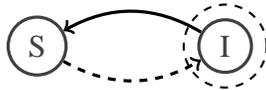
\begin{figure}[h]
\centering
		\begin{tikzpicture}[->, thick]
			\begin{scope}[shift={(1,1)}]			
			\node[circle,very thick,draw=black,opacity=0.75] (S) at (.5,-2.5) {S};
			\node[circle,very thick,draw=black,opacity=0.75] (I) at (3,-2.5) {I};
			\path [->,very thick] (S) edge [bend right, dashed] node[below] {} (I);
			\node[draw=black,dashed,fit=(I) ,inner sep=0.1ex,ellipse] (Icirc) {};
			\path [->,very thick] (I) edge [bend right] node[xshift = -0.0cm, yshift = 0.3cm] {} (S);
			\end{scope}
		\end{tikzpicture}
		\caption{The compartmental diagram of $SIS$. The dashed lines indicate external transitions; the solid lines indicate internal transitions.} \label{fig:SIS}
\end{figure}
We distinguish two cases of interest in considering simulations of the process: (i) the rates are chosen such that an endemic (i.e. non-zero infection) steady-state exists in the mean-field approximation, and (ii) the rates are chosen such that the mean-field approximation dies exponentially quickly.  A simulation of case (i) is given in Figure \ref{fig:SIS_sim}; a simulation of case (ii) is given in Figure \ref{fig:SIS_sim_controlled}.

\begin{figure}[h]
	\centering
	\begin{subfigure}[b]{0.5\textwidth}
		\includegraphics[width =\textwidth]{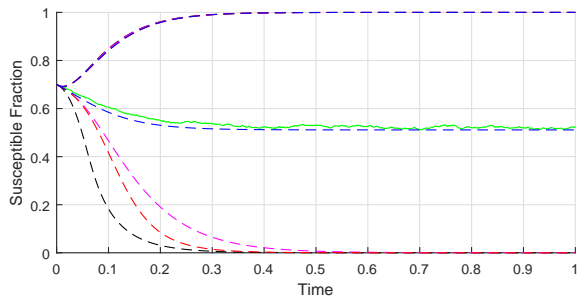}
		\caption{Expected Susceptible Fraction}
	\end{subfigure}
	~
	\begin{subfigure}[b]{0.5\textwidth}
		\includegraphics[width =\textwidth]{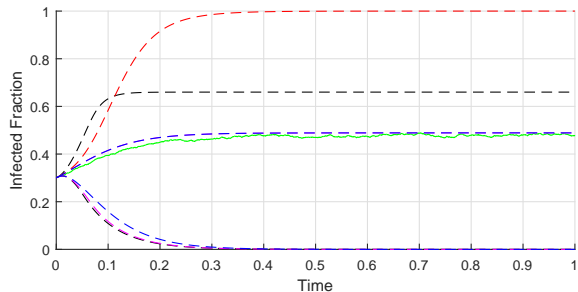}
		\caption{Expected Infected Fraction}
	\end{subfigure}
	\caption{Simulated Dynamics for the SIS Process on a 100-node random graph with connection probability p = 0.2, and rates chosen for the mean-field approximation to attain an endemic steady-state.  The black dashed lines are the trajectories of the bounding system \eqref{eq:bounding}; the magenta dashed lines are the trajectories of \eqref{eq:SIS_dyn}; the blue dashed lines are the trajectories of \eqref{eq:SIS_2}; the red dashed lines are the trajectories of \eqref{eq:SIS_3}; the green solid line is generated from a 100-trial Monte Carlo simulation of the $SIS$ process.  Note that all trajectories given are the average compartmental membership probability approximations for the graph.  Note that the trajectories of \eqref{eq:SIS_dyn} are no worse of an approximation than the trajectories of \eqref{eq:bounding} at any time: this can be verified analytically by an application of Corollary \ref{cor:tight}.} \label{fig:SIS_sim}
\end{figure}

The notable qualities of both figures are essentially the same.  It is clear to see that every specified bounding system does, in fact, bound the evolution of the compartmental membership probabilities appropriately.  It is interesting to note that the trajectories generated by \eqref{eq:SIS_2} appear to perform better with respect to $d_{\{t_0,T\}}$ than all other systems, although it is difficult to verify that this is the case analytically: the hypothesis of Corollary \ref{cor:tight} fails.  Figure \ref{fig:SIS_sim_controlled} demonstrates that the specified bounding systems can perform quite well in some cases: the gap between the upper- and lower- bounding trajectories is no larger than $0.1$ at any point in the system's evolution.

\begin{figure}[h]
	\centering
	\begin{subfigure}[b]{0.5\textwidth}
		\includegraphics[width =\textwidth]{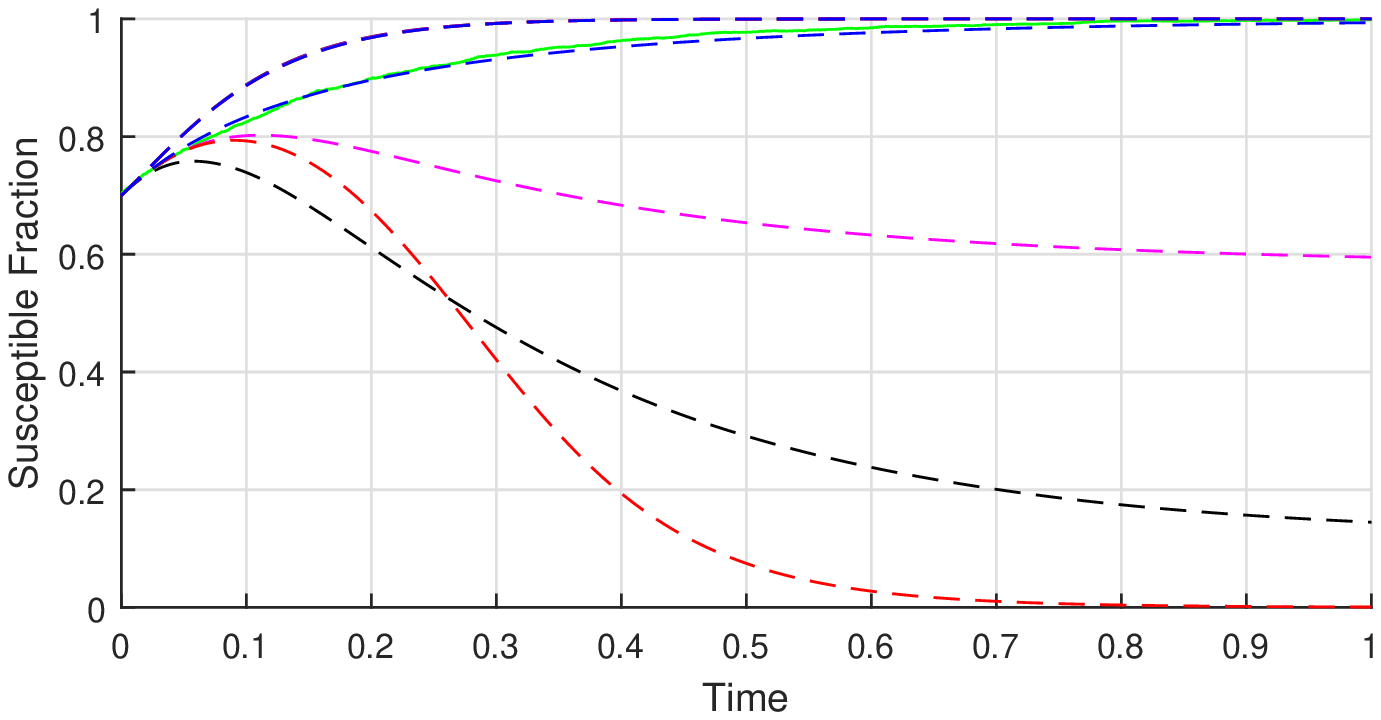}
		\caption{Expected Susceptible Fraction}
	\end{subfigure}
	~
	\begin{subfigure}[b]{0.5\textwidth}
		\includegraphics[width =\textwidth]{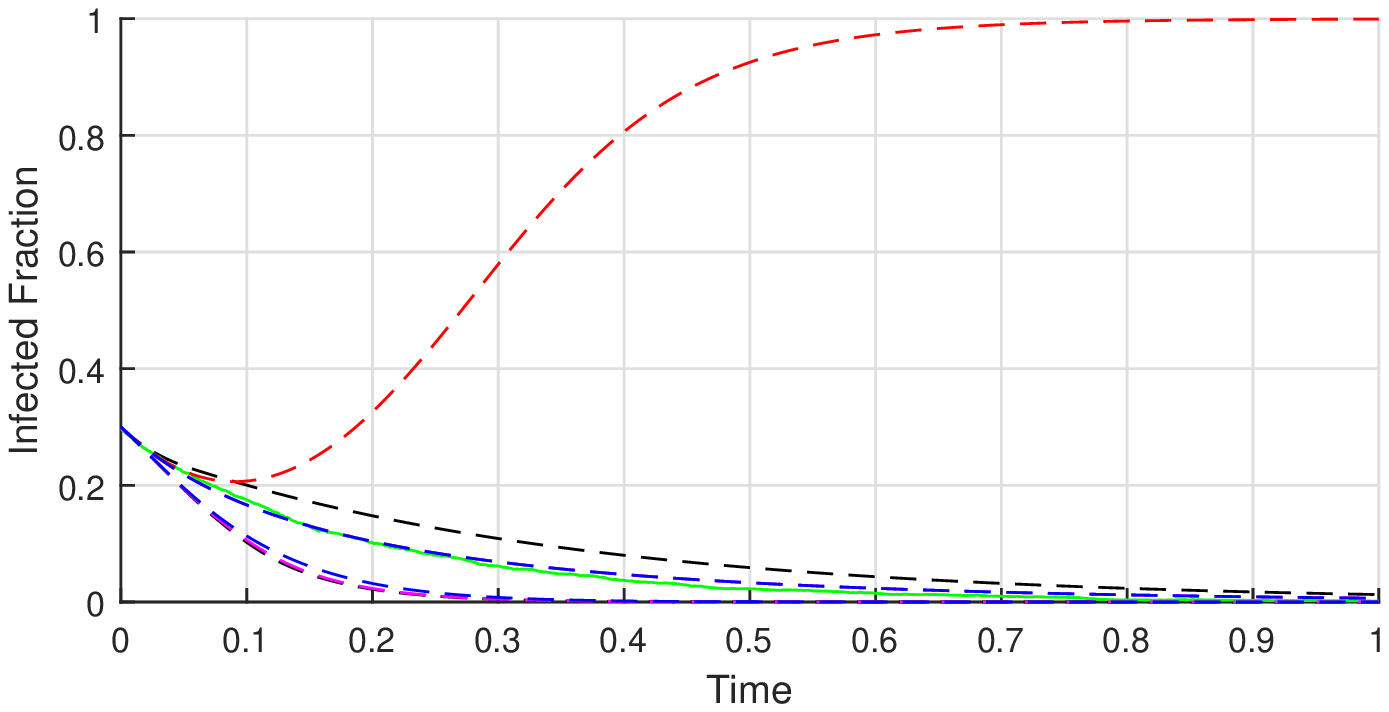}
		\caption{Expected Infected Fraction}
	\end{subfigure}
	\caption{Simulated Dynamics for the SIS Process on a 100-node random graph with connection probability p = 0.2, and rates chosen for the mean-field approximation to decay exponentially.  The black dashed lines are the trajectories of the bounding system \eqref{eq:bounding}; the magenta dashed lines are the trajectories of \eqref{eq:SIS_dyn}; the blue dashed lines are the trajectories of \eqref{eq:SIS_2}; the red dashed lines are the trajectories of \eqref{eq:SIS_3}; the green solid line is generated from a 100-trial Monte Carlo simulation of the $SIS$ process.  Note that all trajectories given are the average compartmental membership probability approximations for the graph.  Note that the trajectories of \eqref{eq:SIS_dyn} are no worse of an approximation than the trajectories of \eqref{eq:bounding} at any time: this can be proven by application of Corollary \ref{cor:tight}} \label{fig:SIS_sim_controlled}
\end{figure}

\subsection{$SI_{1}S I_{2}S$} \label{subsec:SISIS}
The Susceptible - Infected 1 - Susceptible - Infected 2 - Susceptible ($SI_{1}SI_{2}S$) process (pictured in Figure \ref{fig:SISIS}; see, e.g., \cite{sahneh2014competitive}) epidemic process has $|\comps| = 3$, where the three possible compartments an agent may take are `Susceptible,' ($S$), `Infected 1,' ($I_{1}$) and `Infected 2,' ($I_{2}$).   We allow the transitions $\{S \goesto I_{1}\}$ to be external with rate $\lambda_i^{\{S \goesto I_{1}\}}(t) = \sum_{j \in V} \mathbbm{1}_{\{x_j(t) = I_{1}\}} \beta_{ij}^{\{I_1;\; S \goesto I_{1}\}}$, the transitions $\{S \goesto I_{2}\}$ to be external with rate $\lambda_i^{\{S \goesto I_{2}\}}(t) = \sum_{j \in V} \mathbbm{1}_{\{x_j(t) = I_{2}\}} \beta_{ij}^{\{I_2; \; S \goesto I_{2}\}}$, the transitions $\{I_{1} \goesto S\}$ to occur at a rate $\delta_{i}^{\{I_1 \goesto S\}}$ and the transitions $\{I_{2} \goesto S\}$ to occur at a rate $\delta_{i}^{\{I_2 \goesto S\}}$.  We allow no further transitions.  This model allows us to consider a situation in which an agent may belong to one of two factions ($I_1$ or $I_2$), or be undecided $S$.  The standard mean-field approximation for this process can be derived by applying results in \cite{sahneh2013generalized}, and is given as:
\begin{equation} \label{eq:SISIS_mf}
	\begin{aligned}
		\dot{\phi_i}^{\{I_1\}} &= (1 - \phi_i^{\{I_1\}} - \phi_i^{\{I_2\}}) \sum_{j \in V} \beta_{ij}^{\{I_1;S \goesto I_1\}}\phi_j^{\{I_2\}} - \delta_i^{\{I_1 \goesto S\}}\phi_i^{\{I_1\}},\\
		\dot{\phi_i^{\{I_2\}}} &= (1 - \phi_i^{\{I_1\}} - \phi_i^{\{I_2\}}) \sum_{j \in V} \beta_{ij}^{\{I_2;S \goesto I_1\}}\phi_j^{\{I_2\}} - \delta_i^{\{I_2 \goesto S\}}\phi_i^{\{I_2\}},\\
		\phi_i^{\{S\}} &= (1 - \phi_i^{\{I_1\}} - \phi_i^{\{I_2\}}),
	\end{aligned}
\end{equation}
where $\phi_i^{\{I_1\}}$ is the mean-field approximation of the probability that node $i$ is in compartment $I_1$, $\phi_i^{\{I_2\}}$ is the mean-field approximation of the probability that node $i$ is in compartment $I_2$, and $\phi_i^{\{S\}}$ is the mean-field approximation of the probability that node $i$ is in compartment $S$.

Some recent work \cite{watkins2015optimal} has been concerned with designing the transition rates of the network so as to attain a specified steady-state of the mean-field dynamics \eqref{eq:SISIS_mf}, however there is no work to the authors' knowledge relating \eqref{eq:SISIS_mf} to the moment dynamics of $S I_1 S I_2 S$.  Hence, it is interesting to consider simulations of the process in two cases: (i) the mean-field approximation of $I_1$ and $I_2$ equilibrate to an endemic state, and (ii) the mean-field approximation of $I_1$ decays exponentially to $0$, while the $I_2$ is remains uncontrolled.  Note that case (ii) is the control problem studied in \cite{watkins2015optimal}.

\begin{figure}[h]
\centering
\begin{tikzpicture}
	\node[circle,very thick,draw=black,opacity=0.75] (IA) at (-2,0) {$I_1$};
	\node[circle,very thick,draw=black,opacity=0.75] (IB) at (2,0) {$I_2$};
	\node[circle,very thick,draw=black,opacity=0.75] (S) at (0,0) {S};
	\node[draw=black,dashed,fit=(IA) ,inner sep=0.1ex,ellipse] (I1circ) {};
	\node[draw=black,dashed,fit=(IB) ,inner sep=0.1ex,ellipse] (I2circ) {};
	
	\path [->,very thick,dashed] (S) edge [bend left] node[xshift=-0.2cm,yshift=-0.3cm] {} (IA);
	\path [->,very thick] (IA) edge [bend left] node[xshift=0.05cm, yshift=0.3cm] 
	{} (S);
	\path [->,very thick,dashed] (S) edge [bend right] node[xshift=0.2cm,yshift=-0.3cm] {} (IB);
	\path [->,very thick] (IB) edge [bend right] node[xshift=0.05cm, yshift=0.3cm] 
	{} (S);
\end{tikzpicture}
\caption{The compartmental diagram of $S I_1 S I_2 S$. The dashed lines indicate external transitions; the solid lines indicate internal transitions.} \label{fig:SISIS}
\end{figure}

\begin{figure}[h]
	\centering
	\begin{subfigure}[b]{0.5\textwidth}
		\includegraphics[width =\textwidth]{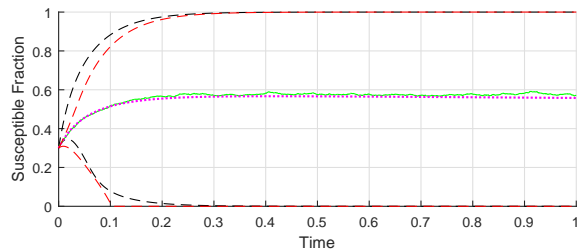}
		\caption{Expected Susceptible Fraction}
	\end{subfigure}
	~
	\begin{subfigure}[b]{0.5\textwidth}
		\includegraphics[width =\textwidth]{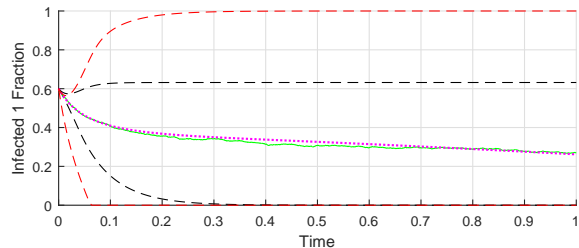}
		\caption{Expected Infected 1 Fraction}
	\end{subfigure}
	~
	\begin{subfigure}[b]{0.5\textwidth}
		\includegraphics[width =\textwidth]{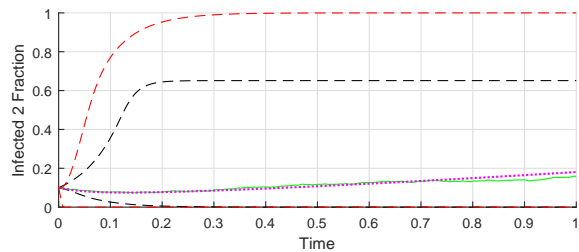}
		\caption{Expected Infected 2 Fraction}
	\end{subfigure}
	\caption{Simulated Dynamics for the $S I_1 S I_2 S$ Process on a 100-node random graph with connection probability p = 0.2 with both mean-field approximations of $I_1$ and $I_2$ surviving in an endemic steady state.  The black dashed lines are the trajectories of the bounding system \eqref{eq:bounding}; the red dashed lines are the trajectories resulting from the application of Corollary \ref{cor:eliminating} to the trajectories of \eqref{eq:bounding}; the magenta dotted lines are the trajectories of the mean-field approximation \eqref{eq:SISIS_mf}; the green solid line is generated from a 100-trial Monte Carlo simulation of the process.  Note that all trajectories given are the average compartmental membership probability approximations for the graph.  Note that the mean-field approximation \eqref{eq:SISIS_mf} does not systematically over or under approximate the states.} \label{fig:SISIS_sim} 
\end{figure}

\begin{figure}[h]
	\centering
	\begin{subfigure}[b]{0.5\textwidth}
		\includegraphics[width =\textwidth]{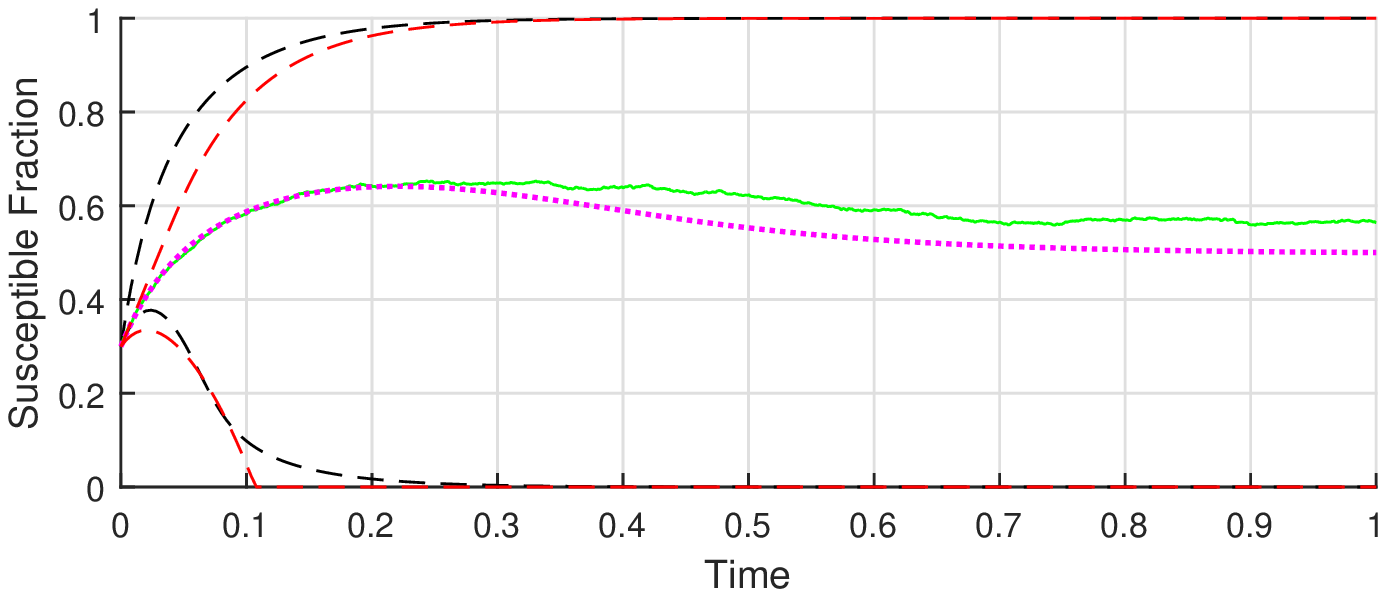}
		\caption{Expected Susceptible Fraction}
	\end{subfigure}
	~
	\begin{subfigure}[b]{0.5\textwidth}
		\includegraphics[width =\textwidth]{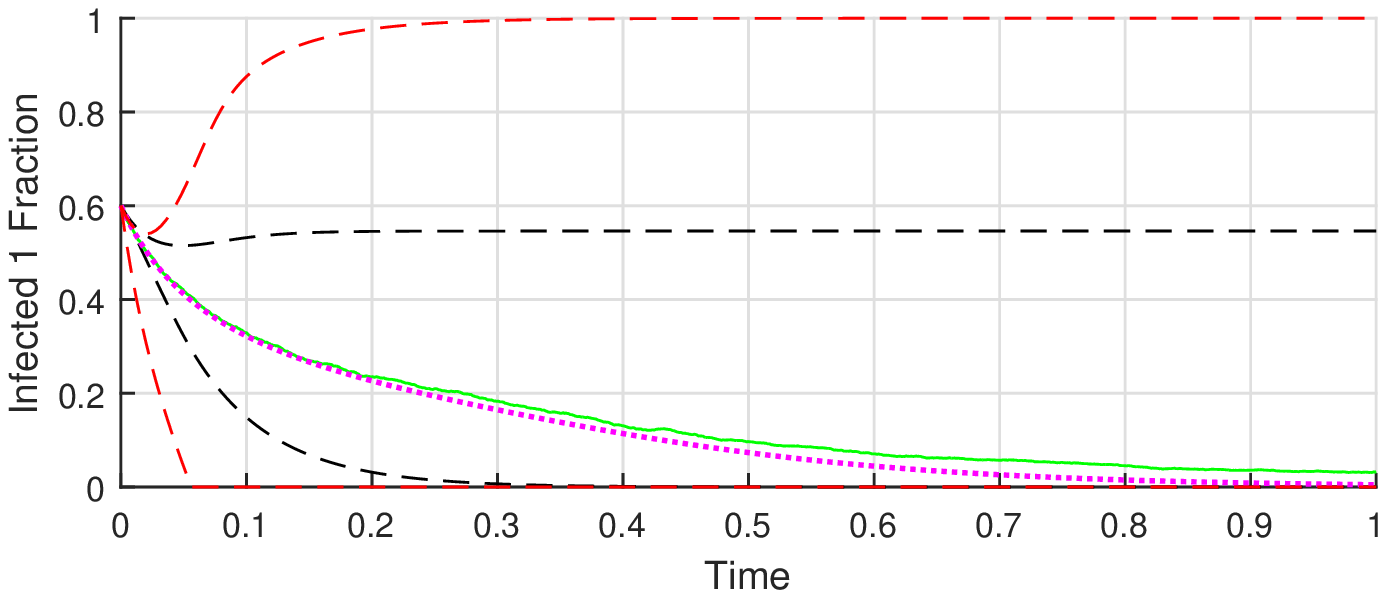}
		\caption{Expected Infected 1 Fraction}
	\end{subfigure}
	~
	\begin{subfigure}[b]{0.5\textwidth}
		\includegraphics[width =\textwidth]{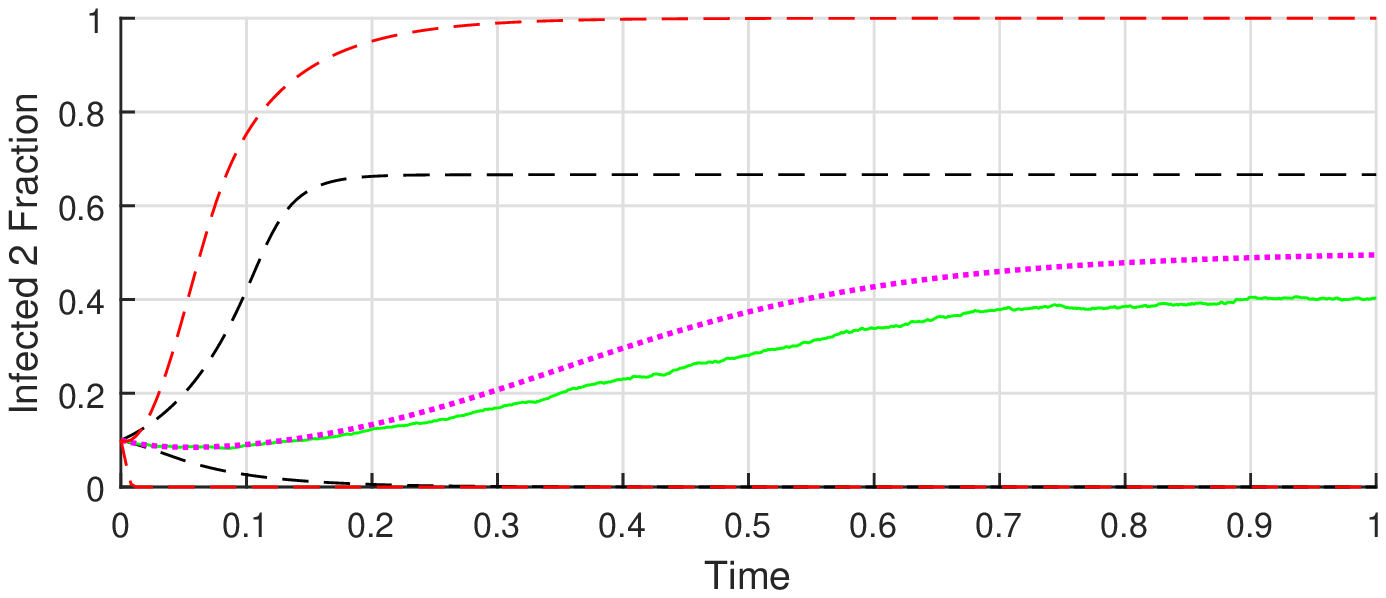}
		\caption{Expected Infected 2 Fraction}
	\end{subfigure}
	\caption{Simulated Dynamics for the $S I_1 S I_2 S$ Process on a 100-node random graph with connection probability p = 0.2 with the mean-field approximation of $I_1$ decaying to $0$ exponentially fast.  The black dashed lines are the trajectories of the bounding system \eqref{eq:bounding}; the red dashed lines plotter are the trajectories resultant from applying Corollary \ref{cor:eliminating} to the trajectories of \eqref{eq:bounding}; the magenta dotted lines are the trajectories of the mean-field approximation \eqref{eq:SISIS_mf}; the green solid line is generated from a 100-trial Monte Carlo simulation of the process.  Note that all trajectories given are the average compartmental membership probability approximations for the graph.  Note that the mean-field approximation \eqref{eq:SISIS_mf} does not systematically over or under approximate the states.} \label{fig:SISIS_sim_controlled} 
\end{figure}

Figure \ref{fig:SISIS_sim} studies case (i); Figure \ref{fig:SISIS_sim_controlled} studies case (ii).  Note that the while the trajectories of \eqref{eq:SISIS_mf} do \emph{appear} to be good approximations (heuristically), they do not systematically over- or under-estimate the simulated trajectory.  Whether or not this is acceptable in any future practice is unclear, however it \emph{does} prevent a designer's ability to make guarantees about system performance, and as such it may be useful to consider the behavior of a bounding system as an alternative to mean-field approximation.  In particular, the control of the mean-field approximation of $S I_1 S I_2 S$ does not imply control of the stochastic process.

\subsection{$SEIV$} \label{subsec:SEIV}
The Susceptible-Exposed-Infected-Vigilant ($SEIV$) process (pictured in Figure \ref{fig:SEIV}; see, e.g., \cite{CN-VMP-GJP:15-TCNS}) model has $|\comps| = 4$, where the possible compartments are `Susceptible,' ($S$), `Infected,' ($I$), `Exposed,' ($E$) and `Vigilant,' ($V$).  The transitions $\{S \goesto V\}$, $\{V \goesto S\}$, $\{E \goesto I\}$, $\{I \goesto V\}$ are specified as internal.  The transition $\{S \goesto E\}$ is an external process, with transition rate $\lambda_{i (t)}^{\{S \goesto E\}} = \sum_j \mathbbm{1}_{\{x_j(t) = E\}}\beta_{ij}^{\{E; \; S \goesto E\}} + \mathbbm{1}_{\{x_j(t) = I \}} \beta_{ij}^{\{I; \; S \goesto E\}}$.  A standard mean-field approximation can be derived by application of results in \cite{sahneh2013generalized}, and can be written as:
\begin{equation} \label{eq:SEIV_mf}
	\begin{aligned}
	\phi_i^{\{S\}} &= (1 - \phi_i^{\{E\}} - \phi_i^{\{I\}} - \phi_i^{\{V\}}),\\
	\dot{\phi}_i^{\{E\}} &= (1 - \phi_i^{\{E\}} - \phi_i^{\{I\}} - \phi_i^{\{V\}}) (\sum_{j \in V} \beta_{ij}^{\{E;S \goesto E\}} \phi_j^{\{E\}} +\beta_{ij}^{\{I ;S \goesto E\}} \phi_j^{\{I\}})- \phi_i^{\{E\}} \delta_i^{\{E \goesto I\}},\\
	\dot{\phi}_i^{\{I\}} &= \delta_i^{\{E \goesto I\}} \phi_i^{\{E\}} - \delta_i^{\{I \goesto V\}} \phi_i^{\{I\}},\\
	\dot{\phi}_i^{\{V\}} &= \delta_i^{\{I \goesto V\}} \phi_i^{\{I\}} + \delta_i^{\{S \goesto V\}}(1 - \phi_i^{\{E\}} - \phi_i^{\{I\}} - \phi_i^{\{V\}}) - \delta_i^{\{V \goesto S \}} \phi_i^{\{V\}},
	\end{aligned}
\end{equation}
where $\phi_i^{\{S\}}$ is an approximation of the probability that node $i$ is a member of compartment $S$, $\phi_i^{\{E\}}$ is an approximation of the probability that node $i$ is a member of compartment $E$, $\phi_i^{\{I\}}$ is an approximation of the probability that node $i$ is a member of compartment $I$, and $\phi_i^{\{V\}}$ is an approximation of the probability that node $i$ is a member of compartment $V$.  Note that \eqref{eq:SEIV_mf} is cannot trivially be shown to be a bounding system; at the least, knowledge of non-negative (or non-positive) correlation between variables must be established before applying our results to derive a comparable set of equations.  However, some recent work \cite{CN-VMP-GJP:15-TCNS} has studied the control of the $SEIV$ process via the mean-field approximation \eqref{eq:SEIV_mf}, so it is interesting to study the behavior of \eqref{eq:SEIV_mf} in comparison to a simulation of the process and our crude bounding system \eqref{eq:bounding}.

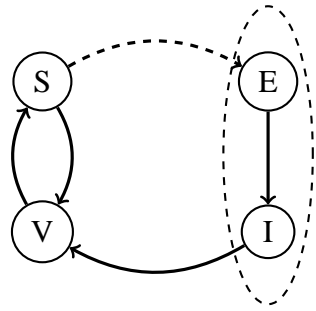
\begin{figure}
\centering
\begin{tikzpicture}[thick,every node/.style={minimum size=.5em},decoration={border,segment length=2mm,amplitude=0.3mm,angle=90}]
				\pgfdeclarelayer{background}
				\pgfdeclarelayer{foreground}
				\pgfsetlayers{background,main,foreground}
  \node[draw,circle,fill=white, thick] (1) at (-1,1) {S};
  \node[draw,circle,fill=white, thick] (2) at (2,1) {E};
  \node[draw,circle,fill=white, thick] (3) at (2,-1) {I};
  \node[draw,circle,fill=white, thick] (4) at (-1,-1) {V};
  \node[draw=black,dashed,fit=(2) (3) ,inner sep=0.1ex,ellipse] (el1) {};
  
  \path[->] (1) edge [very thick,bend left,dashed] (2);
  \path[->] (2) edge [very thick] (3);
  \path[->] (3) edge [very thick,bend left] (4);
  \path[->] (4) edge [very thick,bend left] (1);
  \path[->] (1) edge [very thick, bend left] (4);
  

\end{tikzpicture}
	\caption{The compartmental diagram of $SEIV$. The dashed lines indicate external transitions; the solid lines indicate internal transitions.} \label{fig:SEIV}
\end{figure}

\begin{figure}[h]
	\centering
	\begin{subfigure}[b]{0.4\textwidth}
		\includegraphics[width =\textwidth]{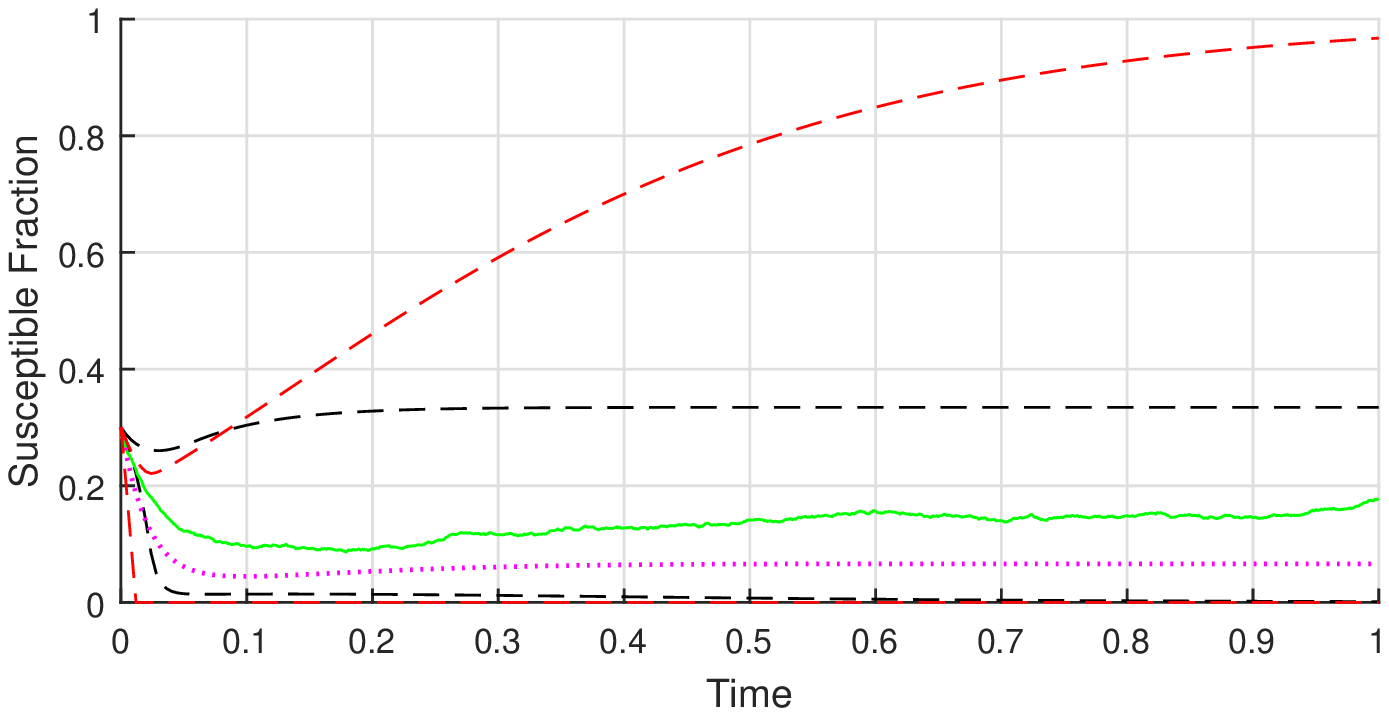}
		\caption{Expected Susceptible Fraction}
	\end{subfigure}
	~
	\begin{subfigure}[b]{0.4\textwidth}
		\includegraphics[width =\textwidth]{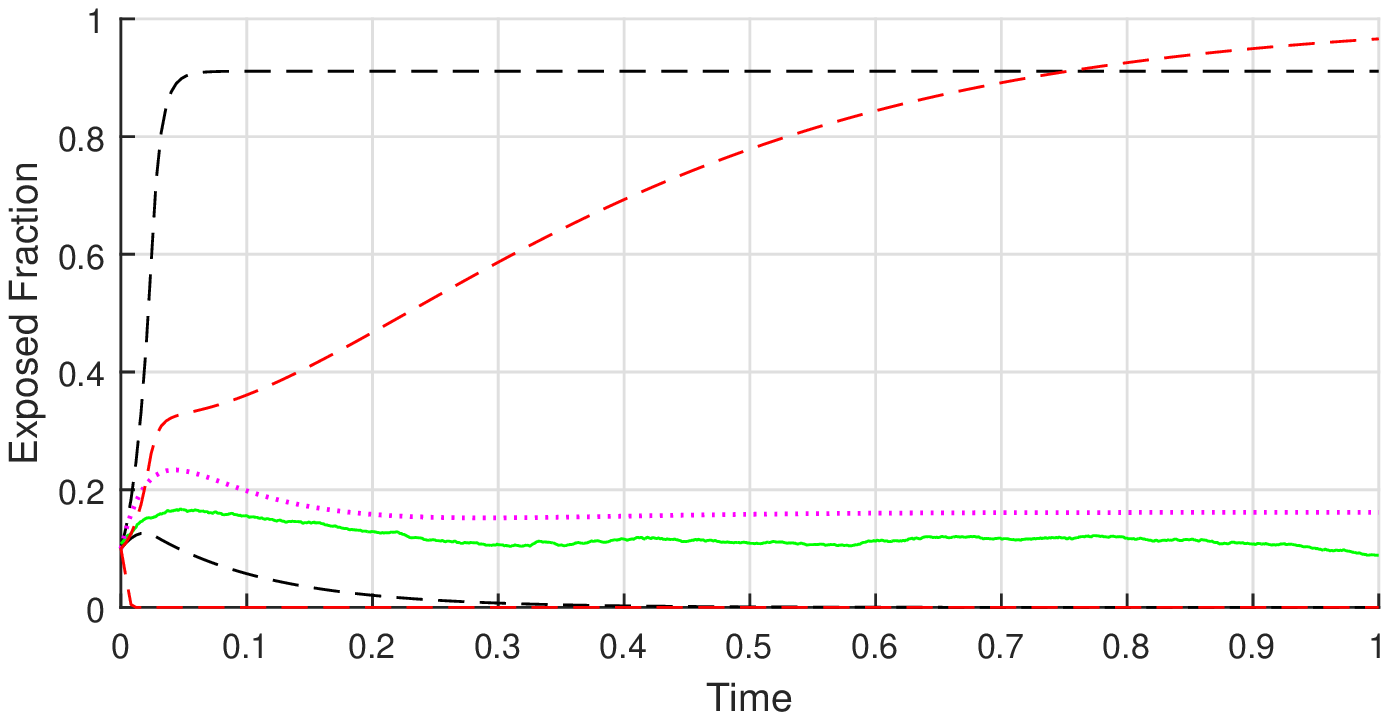}
		\caption{Expected Exposed Fraction}
	\end{subfigure}
	~
	\begin{subfigure}[b]{0.4\textwidth}
		\includegraphics[width =\textwidth]{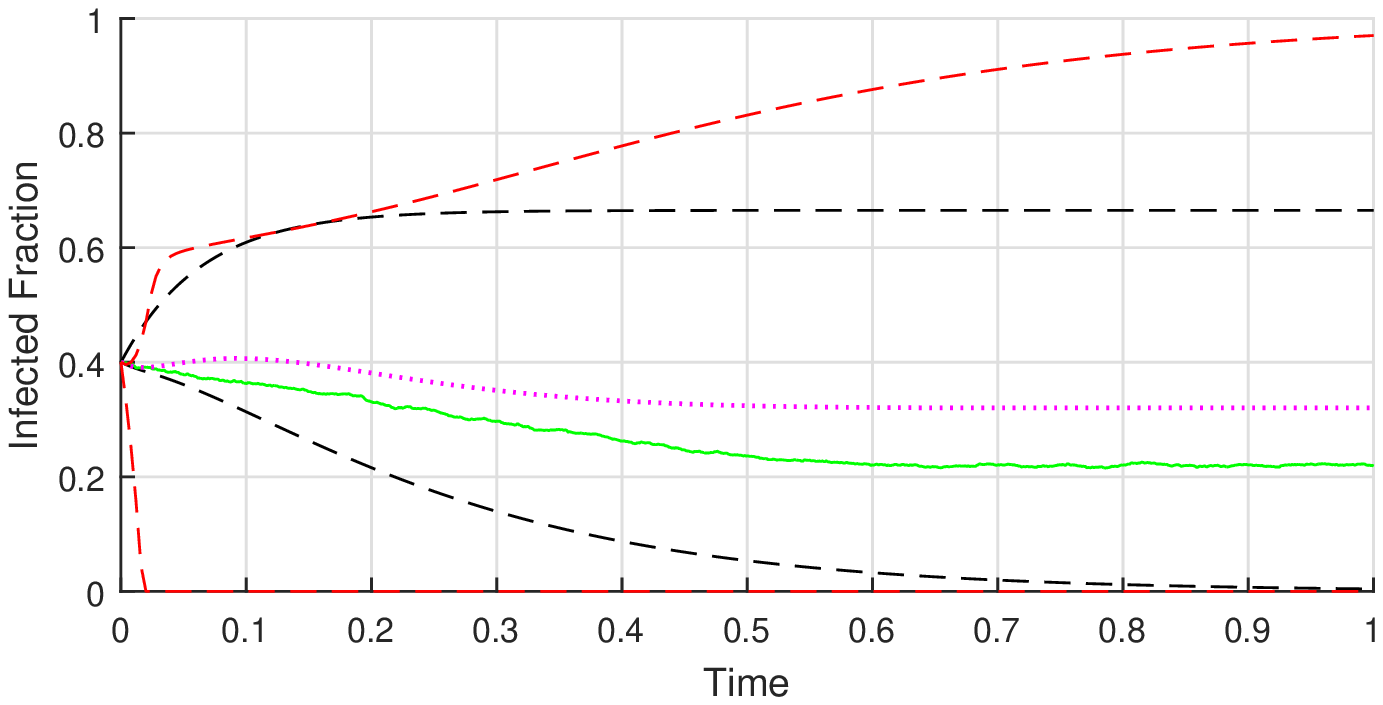}
		\caption{Expected Infected Fraction}
	\end{subfigure}
	~
	\begin{subfigure}[b]{0.4\textwidth}
		\includegraphics[width =\textwidth]{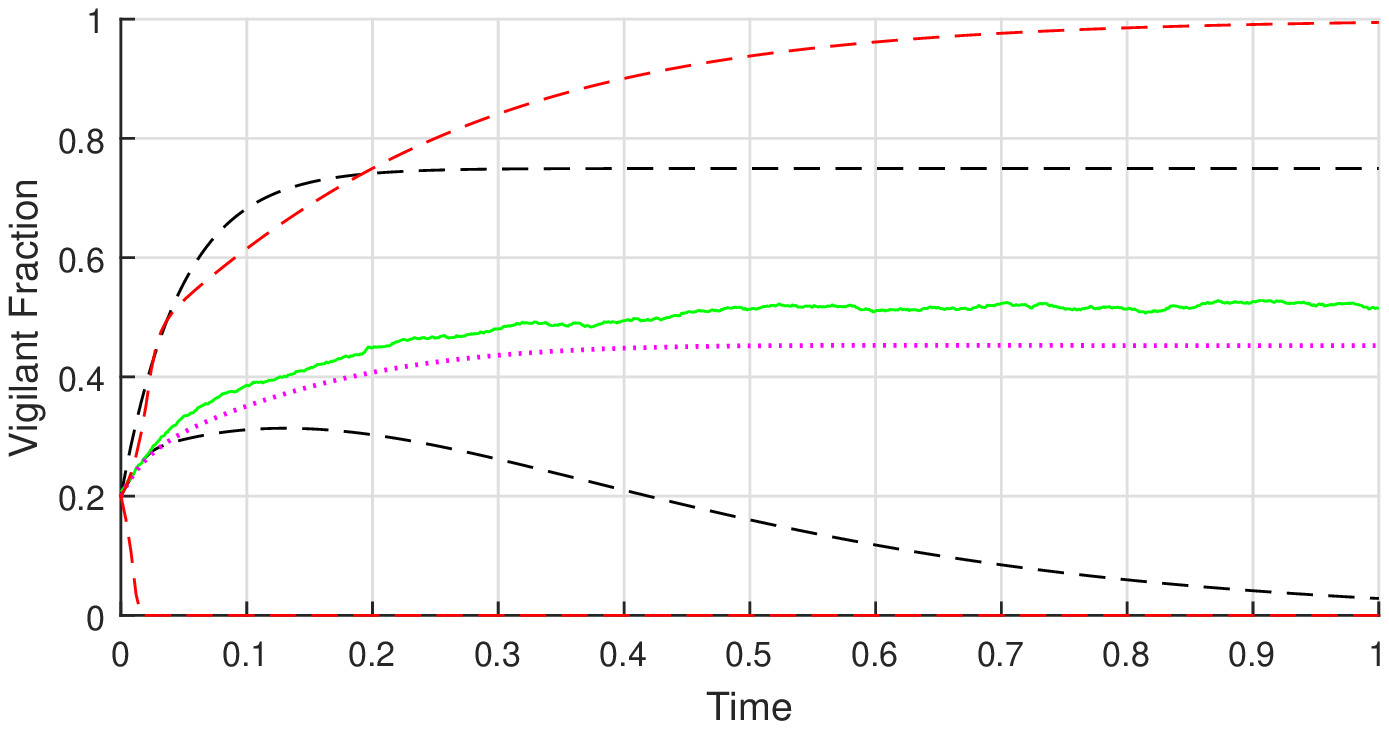}
		\caption{Expected Vigilant Fraction}
	\end{subfigure}
	\caption{Simulated Dynamics for the $S E I V$ Process on a 50-node random graph with connection probability p = 0.2.  The black dashed lines are the trajectories of the bounding system \eqref{eq:bounding}; the red dashed lines are the resulting trajectories from applying Corollary \ref{cor:eliminating} to the trajectories of \eqref{eq:bounding};the magenta dashed lines are the trajectories of the standard mean-field approximation \eqref{eq:SEIV_mf}; the green solid line is generated from a 100-trial Monte Carlo simulation of the process.  Note that all trajectories given are the average compartmental membership probability approximations for the graph.} \label{fig:SEIV_sim}
\end{figure}

\begin{figure}[h]
	\centering
	\begin{subfigure}[b]{0.4\textwidth}
		\includegraphics[width =\textwidth]{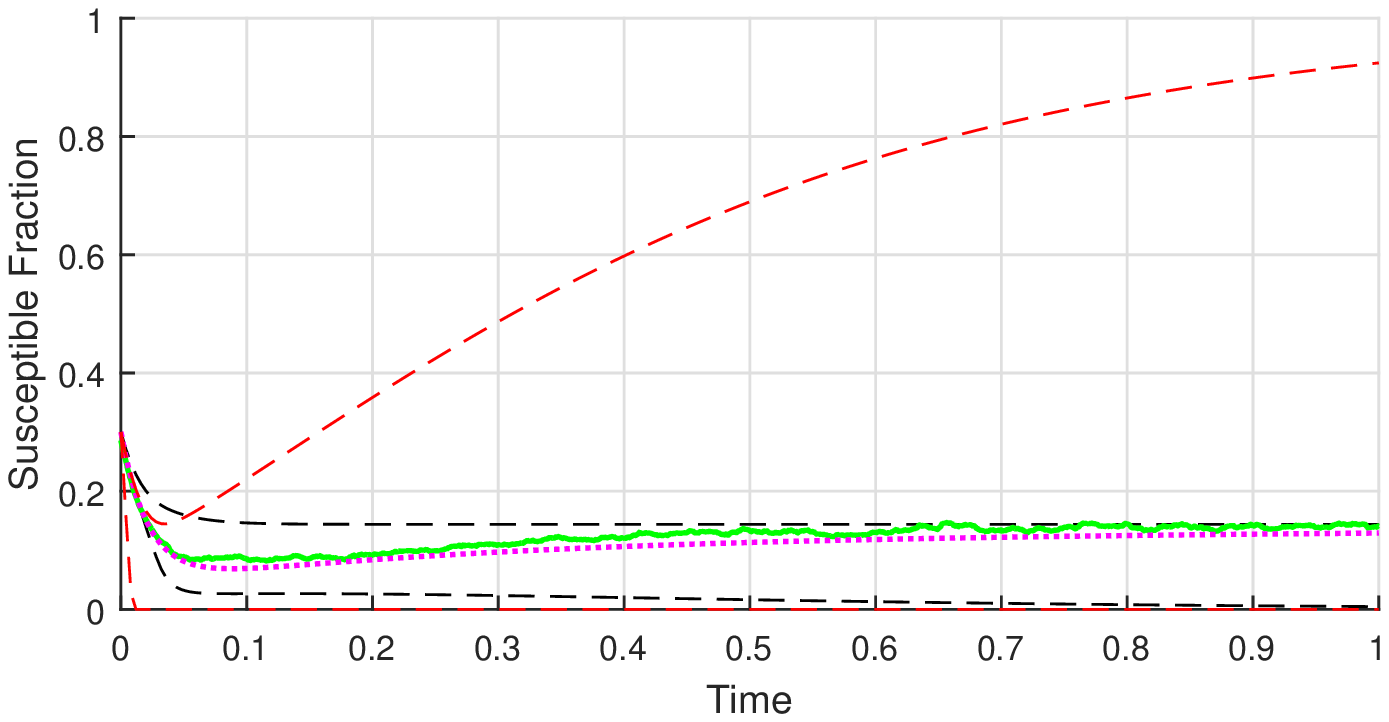}
		\caption{Expected Susceptible Fraction}
	\end{subfigure}
	~
	\begin{subfigure}[b]{0.4\textwidth}
		\includegraphics[width =\textwidth]{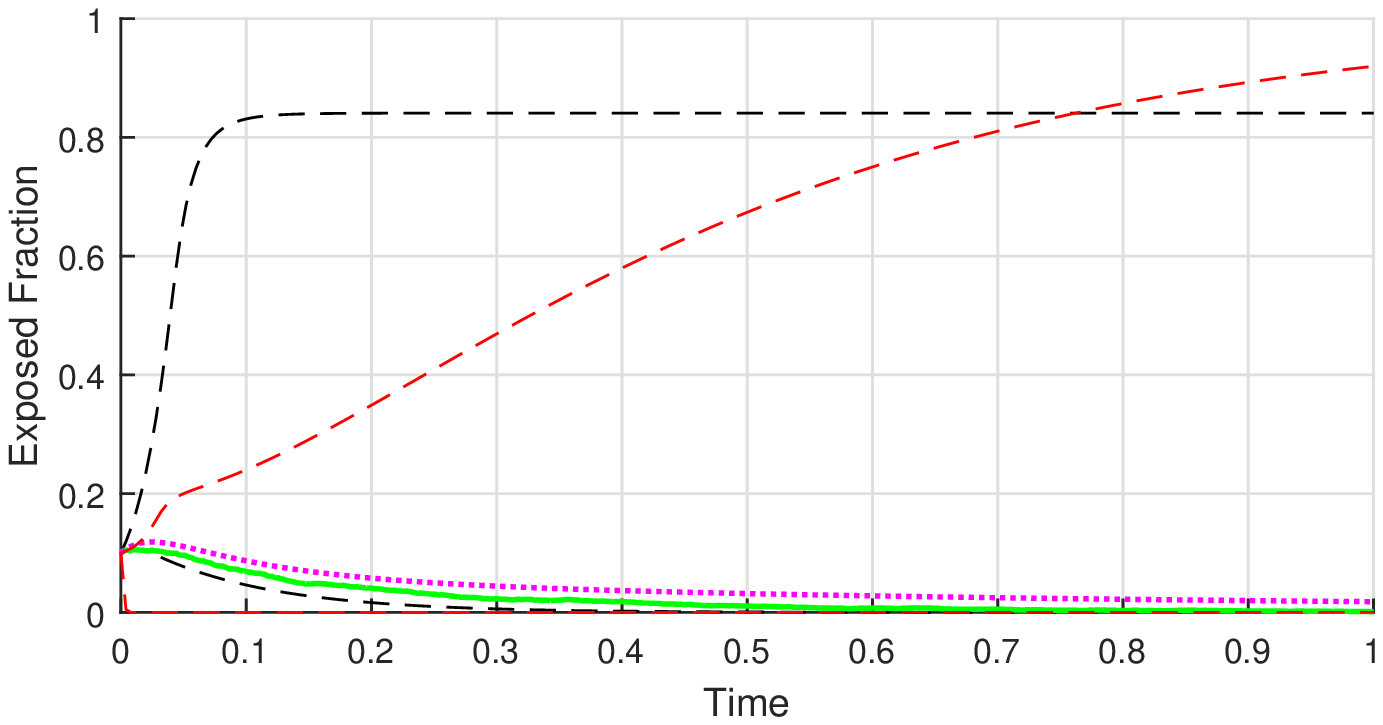}
		\caption{Expected Exposed Fraction}
	\end{subfigure}
	~
	\begin{subfigure}[b]{0.4\textwidth}
		\includegraphics[width =\textwidth]{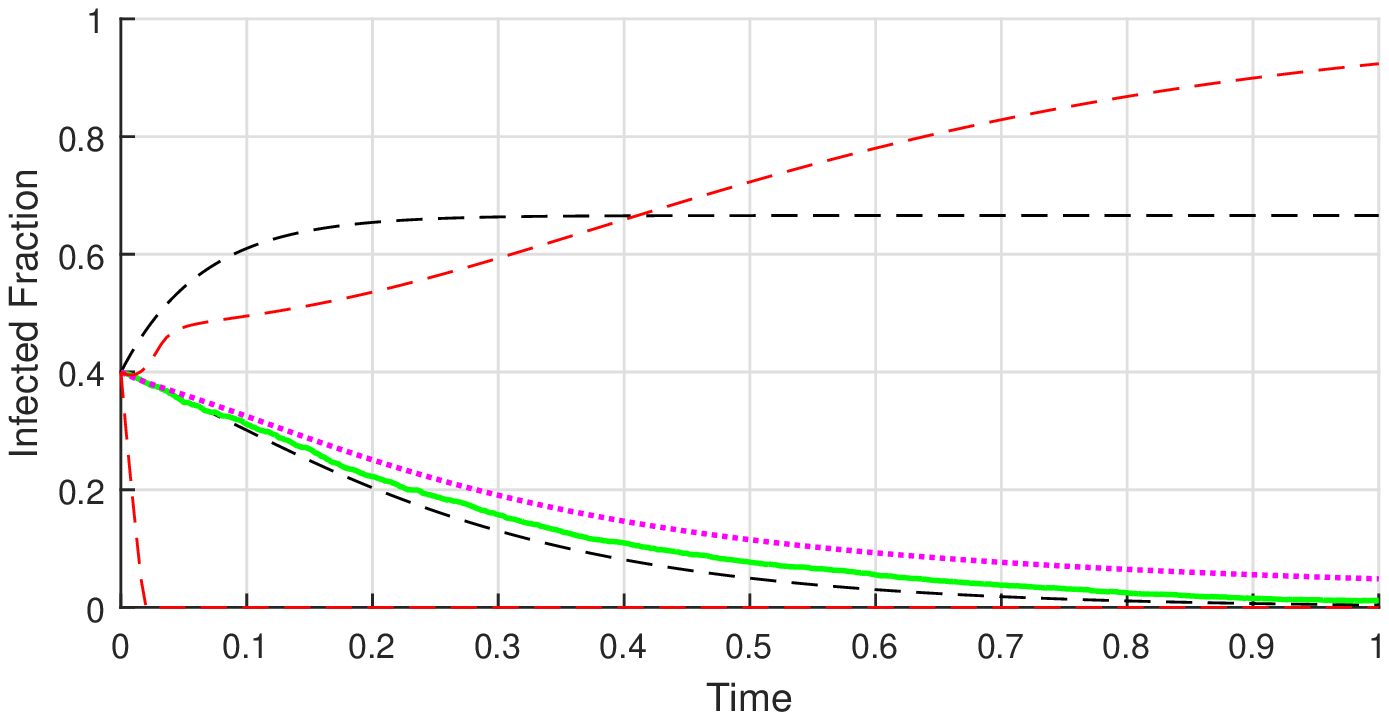}
		\caption{Expected Infected Fraction}
	\end{subfigure}
	~
	\begin{subfigure}[b]{0.4\textwidth}
		\includegraphics[width =\textwidth]{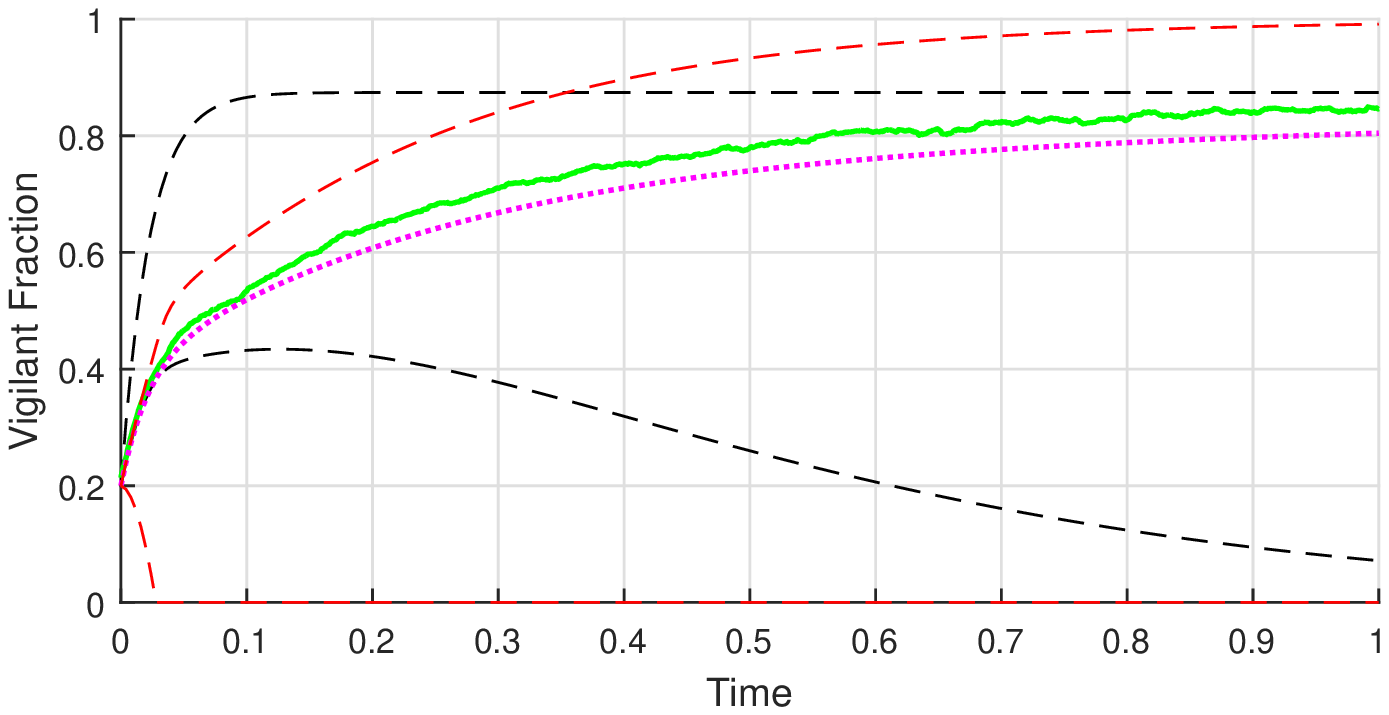}
		\caption{Expected Vigilant Fraction}
	\end{subfigure}
	\caption{Simulated Dynamics for the $S E I V$ Process on a 50-node random graph with connection probability p = 0.2.  The black dashed lines are the trajectories of the bounding system \eqref{eq:bounding}; the red dashed lines are the trajectories resulting from an application of Corollary \ref{cor:eliminating} to the trajectories of \eqref{eq:bounding}; the magenta dotted lines are the trajectories of the mean-field approximation \eqref{eq:SEIV_mf}; the green solid line is generated from a 100-trial Monte Carlo simulation of the process.  Note that all trajectories given are the average compartmental membership probability approximations for the graph.} \label{fig:SEIV_sim_controlled}
\end{figure}

We again distinguish two cases of interest: (i) both the mean-field approximations $\phi_i^{\{E\}}$ and $\phi_i^{\{I\}}$ exist in an endemic state, and (ii) the rates are chosen such that the mean-field approximations of $\phi_i^{\{E\}}$ and $\phi_i^{\{I\}}$ exponentially decay to $0$ (as was studied in \cite{CN-VMP-GJP:15-TCNS}).  Case (i) is studied in Figure \ref{fig:SEIV_sim}; Case (ii) is studied in Figure \ref{fig:SEIV_sim_controlled}.  

Here again, it is interesting to note that the mean-field dynamics \eqref{eq:SEIV_mf} are a tighter approximation to the simulated trajectory of the process than any of the trajectories of the bounding system, however they do \emph{not} systematically over- or under- approximate the simulated trajectory.  However, the simulations exhibit the property that the probabilities of membership in compartments $E$ and $I$ are over-approximated, whereas the probabilities of membership in compartments $S$ and $V$ are under-approximated, although this is a heuristic observation.  This suggests that if a method for proving non-negative or non-positive correlation between process variables is developed, we may be able to construct a bounding system which behaves much like the mean-field approximation.

We note that application of Corollary \ref{cor:eliminating} exhibits a substantial benefit in the simulation of $SEIV$.  This can be observed by examining the red dashed trajectories (those generated by the application of Corollary \ref{cor:eliminating}) of the probabilities of membership in the ``exposed,'' compartment, which provide a substantial improvement over the black dashed trajectories (those generated by the dynamics \eqref{eq:bounding}).

\section{Summary and Future Work} \label{sec:summary}
Summarized most abstractly, we have shown a method constructing systems which generate trajectories that systematically over- and under-approximate the moments of stochastic processes with moment dynamics involving the expectation of a product of indicator random variables.  Though widely applicable in the analysis of chemical, biological and engineered systems, we have focused our discussion on compartmental spreading processes.  We have shown that a crude, but non-trivial, approximating system always exists for a general class of compartmental spreading process.  We have shown that the crude construction can be improved when knowledge of non-negative or non-positive correlation between pairs of process variables can be shown \emph{a priori}.  

In the case that all process transitions are internal, we recover the exact moment dynamics of the process.  In the cases of $SIS$ and $SIR$, our construction leads to a recovery of the standard mean-field approximation developed in \cite{sahneh2013generalized} when we apply knowledge of non-negative correlation between infection compartment variables.  In the cases of $S I_1 S I_2 S$ and $SEIV$, we see that our crude system works as expected, but the mean-field approximations appear to agree more closely with the simulation of the stochastic process.

The problem of determining whether process variables have non-negative or non-positive covariance is open for general compartmental processes, and is of importance.  In the context of this work, it would allow us to construct better bounding systems of arbitrary spreading processes fitting the paradigm presented herein without performing tedious \emph{ad hoc} analyses.  Though there is work near to this theme performed for the case of $SIS$ and $SIR$ \cite{cator2014nodal}, it is unclear at the time of this writing that the arguments presented therein extend to a more general framework, such as the one presented in this paper.

Though we have not discussed it here, dynamic equations for higher order moments can be approximated by the technique presented here as well.  In the case of compartmental spreading processes, the second-order moment dynamics generally involve the expectation of products of three indicators, and we can use the presented bounding techniques almost directly.  By studying bounding systems of the first and second moments of a process, we can develop rigorous concentration arguments by applying Markov's inequality, which would allow for a rigorous mechanism of forecasting, however we leave further development to future work.
\section*{Acknowledgments}
This work was supported in part by the National Science Foundation grant CNS-1302222 “NeTS: Medium: Collaborative Research: Optimal Communication for Faster Sensor Network Coordination”, IIS-1447470 “BIGDATA: Spectral Analysis and Control of Evolving Large Scale Networks”, and the TerraSwarm Research Center, one of six centers supported by the STARnet phase of the Focus Center Research Program (FCRP), a Semiconductor Research Corporation program sponsored by MARCO and DARPA.

\appendix
\subsection{Construction of Exact Moment Dynamics} \label{app:construction}
From first principles, we derive a system of equations which describes the evolution of the probability that the state of node $i$ at time $t$ (denoted $x_i(t)$) takes the value $c$ as follows:

\begin{equation}
	\begin{aligned}
	\frac{d \mathbb{E} \left[ \mathbbm{1}_{\{x_{i}(t) = \{c\} \}} \right]}{d t} 
	&= \lim_{h \downarrow 0} \frac{\mathbb{E} \left[ \mathbbm{1}_{\{x_{i}(t + h) = c\}} \right] - 
	\mathbb{E} \left[ \mathbbm{1}_{\{x_{i}(t) = c \}} \right]}{h} \\
	&=  \lim_{h \downarrow 0} \frac{\mathbb{E} 
	\left[\sum_{c^\prime \in \comps \setminus \{c\} } \mathbbm{1}_{\{x_{i}(t + h) = c,\; x_{i}(t) = c^\prime \}} \right] 
	+ \mathbb{E} \left[ \mathbbm{1}_{\{x_{i}(t+h) = c, \; x_{i}(t) = c \}} \right]
	- \mathbb{E} \left[ \mathbbm{1}_{\{x_{i}(t) = c \}} \right]}{h}\\
	&= \lim_{h \downarrow 0} \frac{\mathbb{E} 
	\left[\sum_{c^\prime \in \comps \setminus \{c\} } \mathbbm{1}_{\{x_{i}(t + h) = c,\; x_{i}(t) = c^\prime \}} \right] 
	- \mathbb{E} \left[ \mathbbm{1}_{\{x_{i}(t+h) \in \comps \setminus \{c\}, \; x_{i}(t) = c \}} \right]}{h}\\
	&= \lim_{h \downarrow 0} \frac{\mathbb{E} 
	\left[\sum_{c^\prime \in \comps \setminus \{c\} } \mathbbm{1}_{\{x_{i}(t) = c^\prime\}} \rateit{c^\prime \rightarrow c} \right] h
	- \mathbb{E} \left[\sum_{c^\prime \in \comps \setminus \{c\}} \mathbbm{1}_{\{x_{i}(t) = c \}} \rateit{c \rightarrow c^\prime} \right] h + h O(h)}{h}\\
	&= \mathbb{E} 
		 \left[ \sum_{c^\prime \in \comps \setminus \{c\} } \mathbbm{1}_{\{x_{i}(t) = c^\prime\}} \rateit{c^\prime \rightarrow c} \right]
		- \mathbb{E} \left[ \sum_{c^\prime \in \comps \setminus \{c\}} \mathbbm{1}_{\{x_{i}(t) = c \}} \rateit{c \rightarrow c^\prime} \right],
	\end{aligned}
\end{equation}

where we have made the assumption that $$| \mathbb{E}[\sum_{c^\prime \in \comps \setminus \{c\}} \mathbbm{1}_{\{x_i(t+h) = c,x_i(t) = c^\prime \}}] - \mathbb{E}[\sum_{c^\prime \in \comps \setminus \{c\}} \mathbbm{1}_{\{x_i(t) = c^\prime \}} \rateit{c^\prime \goesto c} h] | \leq h O(h),$$ which holds true under mild assumptions, such as the condition $0 \leq \delta_i^{\{c\}} < \infty$ and $0 \leq \beta_{ij}^{\{\tilde{c}; \; c \goesto c^\prime\}} < \infty$
for all $i$, $j$, $c$ and $c^\prime$, which we have assumed in the body of the paper.  These expressions are the Kolmogorov forward equations (see, e.g., \cite{stroock2005introduction}) for the processes we are considering, where we have made a deliberate choice to represent the system from a microscopic (i.e. agent-centric) perspective.

\bibliography{research}
\bibliographystyle{ieeetr}
\end{document}